\documentclass[oneside,english]{amsart}
\usepackage[T1]{fontenc}
\usepackage[latin9]{inputenc}
\usepackage{geometry}
\usepackage{amstext}
\usepackage{amsthm}
\usepackage{amssymb}
\usepackage{epsfig,epstopdf}
\usepackage{graphicx}

\makeatletter
\numberwithin{equation}{section}
\numberwithin{figure}{section}
\theoremstyle{plain}
\newtheorem{thm}{\protect\theoremname}
\theoremstyle{definition}
\newtheorem{defn}[thm]{\protect\definitionname}
\theoremstyle{plain}
\newtheorem{lem}[thm]{\protect\lemmaname}

\usepackage{float}

\newtheorem{theorem}{Theorem}

\usepackage{amsfonts} 
\makeatother

\usepackage{babel}

\title[NLMC with RVE. Bridging separable and non-separable scales.]
{Nonlocal multicontinua with Representative Volume Elements. Bridging separable and non-separable scales}

\author{Eric T. Chung}
\address{Department of Mathematics, The Chinese University of Hong Kong, Shatin, New Territories, Hong Kong SAR, China}

\author{Y. Efendiev} 
\address{Department of Mathematics, Texas A\&M University, College Station, TX 77843, USA}
\email{efendiev@math.tamu.edu}

\author{Wing T. Leung}
\address{ICES, University of Texas, Austin, TX, USA}

\author{M. Vasilyeva}
\address{Institute for Scientific Computation, Texas A\&M University, College Station, TX, USA  \& Multiscale model reduction laboratory, North-Eastern Federal University, Yakutsk, Russia}

\makeatother

\usepackage{babel}
\providecommand{\definitionname}{Definition}
\providecommand{\lemmaname}{Lemma}
\providecommand{\theoremname}{Theorem}

\begin{document}

%\title{Nonlocal multicontinua with Representative Volume Elements. Bridging separable and non-separable scales}

\maketitle

\begin{abstract}

Recently, several approaches for multiscale simulations for problems with
high contrast and no scale separation are introduced. Among them is
nonlocal multicontinua (NLMC) method, which introduces multiple macroscopic
variables in each computational grid. These approaches explore
the entire coarse block resolution
 and one can obtain optimal convergence results
independent of contrast and scales.
However,
these approaches are not amenable to many multiscale
simulations, where the subgrid effects are much smaller
than the coarse-mesh resolution. For example, molecular dynamics of shale gas
occurs in much smaller length scales compared to the coarse-mesh size,
which is of orders of meters. In this case, one can not explore
the entire coarse-grid resolution in evaluating effective properties.
In this paper, we merge the concepts of nonlocal multicontinua methods
and Representative Volume Element (RVE)
 concepts to explore problems with extreme scale separation.
The first step of this approach is to use sub-grid scale (sub to RVE) to
write a large-scale macroscopic system. We call it intermediate scale
macroscale system. In the next step, we couple
this intermediate macroscale system to the simulation grid model,
which are used in simulations.
This is done using RVE concepts, where we relate intermediate
macroscale variables to the macroscale variables defined on
our simulation
coarse grid. Our intermediate coarse model allows formulating macroscale
variables correctly and coupling them to the simulation grid.
We present the general concept of our approach and present details
of single-phase flow. Some numerical results are presented. For nonlinear
examples, we use machine learning techniques to compute macroscale parameters.

\end{abstract}
\section{Introduction}

In recent years, many multiscale methods have been developed for solving
challenging problems with multiple scales. Some important classes
of multiscale problems include problems with scale separation and
problems without scale separation and high contrast. For problems
with scale separation, approaches such as homogenization \cite{Jikov91},
heterogeneous multiscale methods \cite{ee03}, equation free \cite{rk07},
and so on, are developed. These approaches explore Representative Volume
Element (RVE) computations and use them to compute effective properties.
To demonstrate the main idea of these approaches, we consider
\[
-div\kappa(x,\nabla u)=f,
\]
where $\kappa(x,\cdot)$ has a scale separation. The computational
domain is divided into coarse blocks and at each coarse block
(see Figure \ref{fig:ill1}),
effective property is computed by solving local problems in each
RVE. These local problems are typically formulated as
\[
-div\kappa(x,\nabla \mathcal{N})=0
\]
subject to $\mathcal{N}=\xi\cdot x$, for all $\xi$. The effective
flux is computed as the average $\kappa^*(\xi) = \langle \kappa(x,\nabla \mathcal{N})\rangle$, where the average is taken over RVE.

The multiscale methods for problems without scale separation has been
an active area of research. These methods typically explore the entire
coarse block and some nearby regions to biuld effective properties.
Some original methods in this direction include Multiscale Finite
Element Method (MsFEM) \cite{hw97},
Generalized Multiscale Finite Element Method
(GMsFEM)
\cite{GMsFEM13,gao2015generalized,chung2018fast,WaveGMsFEM,chung2015goal,MixedGMsFEM}, 
Multiscale Finite Volume \cite{hkj12,jennylt03,jennylt05},
Constraint Energy Minimizing GMsFEM (CEM-GMsFEM) \cite{chung2018constraintmixed}, Nonlocal
Multicontinua Approaches (NLMC) \cite{NLMC},
metric-based upscaling \cite{oz06_1}, Heterogeneous Multiscale Method \cite{ee03,abe07}, LOD \cite{henning2012localized}, equation free approaches \cite{rk07,srk05,skr06}, computational continua \cite{fish2005multiscale,fish2010computational,fafalis2018computational}, hierarchical multiscale method \cite{hs05,brown2013efficient,tan2019high}, homogenization-based approaches \cite{cances2015embedded, le2014multiscale,fu2019edge, le2014msfem, chen2019homogenization, chen2019homogenize,salama2017flow},
and so on.
In this paper, we focus on high-contrast problems, where
approaches CEM-GMsFEM \cite{chung2018constraintCMAME,chung2018constraintmixed} 
and NLMC \cite{NLMC} are used to achieve an optimal convergence
independent of contrast and scales. These approaches use multiple
macroscopic variables and oversampling regions to construct coarse-grid
equations. Next, we briefly describe these approaches before
giving main details of our proposed methods, which use concepts
of RVE upscaling and multicontinua nonlocal upscaling.

We briefly describe nonlocal multicontinua approaches.
In these approaches, we first identify macroscopic variables
in each coarse block via local spectral decomposition.
We denote them by $U_i^{(j)}$, where $i$ is coarse-grid block
and $j$ is macroscopic variable in this coarse block.
These macroscopic variables typically represent averages of solutions
over some regions, which can not be localized, such as high-contrast channels
in porous media applications. In the second step, we construct downscaled
maps from macroscopic variables to the fine grid in the region
of influence (typically oversampled regions),
\[
\mathcal{R}:U_i^{(j)} \rightarrow u_f.
\]
Once the maps from macroscopic variables to the fine-grid field are
identified, we seek macroscopic solution (values of macroscopic variables
$U_i^{(j)}$)
such that the downscaled solution solves the global problem in a weak sense.
We will give more detailed description later on.

The use multiple macroscopic variables are critical in multiscale simulations
\cite{NLMC,fafalis2018computational}.
However, our known approaches use entire coarse resolution to compute
macroscale variables, which are not feasible for many applications. For
example, in shale gas applications, gas dynamics is described by
molecular dynamics of multi-component gas particles. The local
simulations are only possible in small RVEs; however, one needs to perform
large-scale simulations for predicting flow in reservoirs. In this paper,
we couple RVE simulations and nonlocal multicontinua approaches to develop
efficient numerical simulations, where we can
partly explore  the scale separation ideas. These problems can occur in many applications, where intermediate
scales are used
to get very large systems and then, explore RVE concepts. Next, we
briefly describe these ideas.

We assume that there are three macroscopic scales. The first scale is denoted
by $h$ and can be regarded as a scale, where we apply nonlocal multicontinua
approach and write down an intermediate macroscale equations for $U_i^{(j)}$,
\[
G_h(U_i^{(j)})=0.
\]
These equations are nonlocal and  very expensive to solve.
In the next step, we use RVE ideas to connect these variables to
macroscale variables $\overline{U}_i^{(j)}$ defined on $H$-size grid, where
we perform computations. We introduce an intermediate coarse-mesh scale,
RVE scale, denoted by $H_{RVE}$, and assume $h\ll H_{RVE} \ll H$.
In the next step, we use RVE computations to connect  ${U}_i^{(j)}$
to $\overline{U}_i^{(j)}$,
\[
 {U}_i^{(j)}=\mathcal{R} (\overline{U}).
\]
This is done via solving RVE problems subject to some constraints that use
$\overline{U}$. Once we define the map, we perform quadrature of macroscale
equations using RVE cells (cf., \cite{ee03}).
%This assumes some type of periodicity.

In the proposed approaches, we make two major assumptions.
Though the equations on $h$-scale are rigorous, their connections to
$H$-scale equations require (1) periodicity (2) identifying macroscale
variables. As for periodicity assumptions, our approaches are similar to
existing methods, such as HMM, equation free, and so on. However,
defining macroscale variables on $h$-scale and using similar
macroscopic variables
on $H$-scale is one of main advantages, which allow introducing
macroscale variables in a rigorous fashion.

\begin{figure}
\centering
\includegraphics[width=6in]{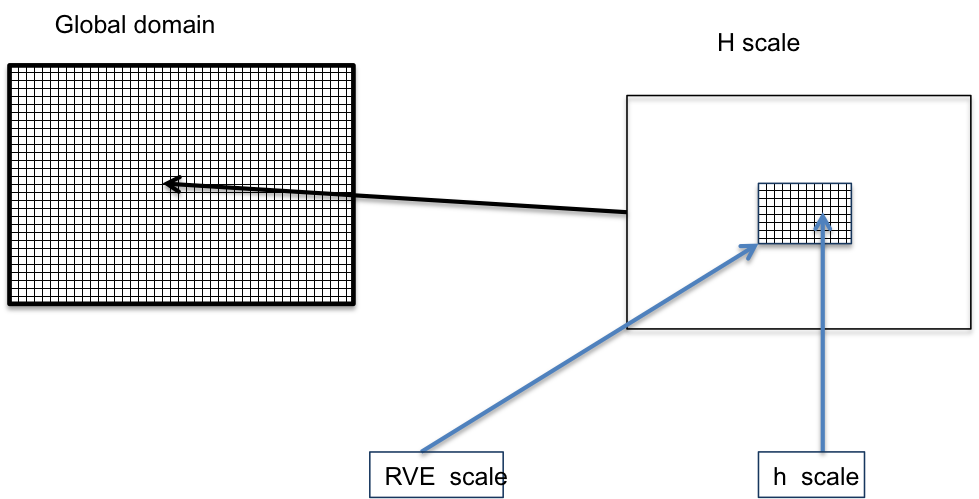}
\caption{Schematic description of the grids}
\label{fig:ill1}
\end{figure}

In the paper, we give an overview of our approach. Our approach
relies on NLMC approach, and for this reason, we first present
this approach. We describe steps of our approach and then give a detailed
study for linear system. We present some convergence result under certain
assumptions.
We also obtain a PDE description of our macroscale equations, which has
integro-differential form.
 We present some numerical results.

The paper is organized as follows. In Section \ref{sec:gen},
 we present a general
concept. Section \ref{sec:detail} is devoted to detaied studies of the linear
case. Finally, in Section \ref{sec:num}, we present some numerical results.

\section{General concept}
\label{sec:gen}

\subsection{NLNLMC on RVE-scale}

We will first follow \cite{chung2018nonlinear,leung2019space} and
present  nonlinear NLMC on RVE scale using the following model nonlinear problem
\begin{equation}
\label{eq:non1}
M U_t + \nabla \cdot G(x, t, U)=g,
\end{equation}
where
$G$ is a nonlinear operator that has a multiscale dependence with respect to
space (and time, in general) and $M$ is a linear operator.
In the above equation, $U$ is the solution and $g$ is a given source term.
In this part, we use the approach proposed in \cite{chung2018nonlinear,leung2019space} on RVE scale, which
is expensive and then apply homogenization idea.
%We summarize these concepts (see \cite{xxx} for details).

\begin{itemize}

\item {\bf The choice of continua}

The continua serves as our macroscopic variables in each coarse element. Our approach uses
a set of test functions to define the continua.
To be more specific, we consider a coarse element $K_i$.
We will choose a set of test functions $\{ \psi_i^{(j)}(x,t) \}$ to define our continua,
where $j$ denotes the $j$-th continuum.  Using these test functions, we can define
our macroscopic variables as
$$
U_i^{(j)} = \langle \langle U, \psi_i^{(j)} \rangle\rangle
$$
where $\langle \langle \cdot,\cdot \rangle\rangle$ is a space-time inner product.

%Define test functions  $\psi_i^{(j)}(x,t)$ that identify continua
%for space-time

\item {\bf The construction of local downscaling map}

Our upscale model uses a local downscaling map to bring microscopic information to the coarse grid model.
The proposed downscaling map is a function defined on an oversampling region subject to some constraints
related to the macroscopic variables.
More precisely, we consider a coarse element $K_i$,
and an oversampling region $K_i^+$ such that $K_i \subset K_i^+$. Then we
find a function $\phi$ by
solving the following local problem
\begin{equation}
\label{eq:non11}
M \phi_t + \nabla \cdot G(x,t, \phi)=\mu, \quad\text{in } K_i^+.
\end{equation}
The above equation (\ref{eq:non11}) is
solved subjected to constraints defined by the following functionals
\[
I_\phi(\psi_i^{(j)}(x,t)).
\]
This constraint fixes some averages of $\phi$ with respect to
$\psi_i^{(j)}(x,t)$.
We remark that the function $\mu$ serves as the Lagrange multiplier for the above constraints. This local solution
 builds a downscaling map
\[
\mathcal{F}_i^{ms}:I_\phi(\psi_i^{(j)}(x,t))\rightarrow \phi.
\]

\item  {\bf The construction of coarse scale model}

We will construct the coarse scale model using the test functions $\{ \psi_i^{(j)}(x,t) \}$
and the local downscaling map. Our upscaling solution $U^{ms}$ is defined
as a combination of the local downscaling maps.
%In particular, we define
%\begin{equation}
%U^{ms} = \sum_i \chi_i \mathcal{F}_i^{ms}(U_i^{(j)}).
%\end{equation}
To compute $U^{ms}$, we use the following variational formulation
\begin{equation}
\label{eq:nonlinearNLMC}
\langle  \langle M U^{ms}_t + \nabla \cdot G(x,t, U^{ms}),\psi_i^{(j)} \rangle \rangle= \langle  \langle g,\psi_i^{(j)} \rangle \rangle.
\end{equation}
The above equation (\ref{eq:nonlinearNLMC}) is our coarse scale model.
%where $U^{ms}$ is global downscaling corresponding to $U$.

\end{itemize}

\subsection{RVE-based NLNLMC}

We denote by $H$ the coarse-mesh size, where the final computations
are performed.
We denote by $h$ a scale, where we write nonlocal
multicontinua equations; however,
they are very large to solve and will be reduced to $H$-scale.
We denote $H_{RVE}$, the scale of RVE and it is assumed $H\gg H_{RVE}\gg h$.
There is also very fine grid,
which is subgrid of $h$.

\begin{itemize}

\item First, we note that in NLNLMC, we find ${U}_i^{(j)}$ (on $h$-scale) such that
\begin{equation}
\label{eq:nonlinearNLMC_1}
\langle  \langle M U^{ms}_t + \nabla \cdot G(x,t, U^{ms}),\psi_i^{(j)} \rangle \rangle= \langle  \langle g,\psi_i^{(j)} \rangle \rangle,
\end{equation}
where $U^{ms}$ depends on  ${U}_i^{(j)}$, which are defined on $h$-scale.
This equation is very large and we will only use RVE-based solution.

\item Our second goal is to use RVE concept and reduce the dimension of
${U}_i^{(j)}$.
We introduce a coarse-grid
homogenized solution and denote it by $\overline{U}^H$ (defined on $H$-scale)
and write its finite element
expansion
\[
\overline{U}^H=\sum_{i,j}\overline{U}_i^{(j)} \overline{\Phi}_i^{j},
\]
where $\overline{\Phi}_i^{j}$ are standard basis functions, for example,
piecewise linear on $H$-scale.
We seek a reduced map
\[
 {U}_i^{(j)}=\mathcal{R}_H(\overline{U}).
\]
This map can be local or nonlocal, in general. Our construction
is based on homogenization ideas and uses local maps, which we
introduce next. Using these local maps, the quadrature can be approximated
on RVEs. In a linear case, $\mathcal{R}$ is a matrix of the sizes corresponding
to $H^{-d}$ and $h^{-d}$.

In the simplest approach, we will use $\mathcal{R}_H$ to be
$L^2$ projection of $\overline{U}^H$ (i.e., the averages of
$\overline{U}^H$ on $h$-scale mesh).

\item There are various ways to use homogenization ideas to construct reduced map.
Here, we consider some of them, which differ in a way we impose constraints.

{\it Approach 1.} In this approach, we solve RVE-based local problem
with constraints given by $\overline{U}^H$ on each RVE cell to define a map
$\mathcal{R}_h^K$, which is local for each coarse block.

{\it Approach 2.} In this approach, we solve local problems on RVE ($H_{RVE}$-scale)
subject to boundary conditions $U^H$ (to be determined)
\begin{equation}
\label{eq:non12}
M \phi_t + \nabla \cdot G(x,t,\phi )=\mu, \quad\text{in } K_i^{RVE}
\end{equation}
subject to boundary condition and initial condition at $t=t_n$
\[
\phi=\mathcal{R}_H(\overline{U}^H),
\]
where $\phi=\mathcal{R}_h(\mathcal{R}_H(\overline{U}^H))$.
This equation gives a map between the local solution and $\mathcal{R}_H(\overline{U}^H)$. Next,
we define
\[
\langle \phi,\psi_i^{(j)}\rangle,
\]
and the map
\[
\langle \phi,\psi_i^{(j)}\rangle=\mathcal{R}_H(\overline{U}_i^{(j)}).
\]
The map $\mathcal{R}_h$ can be regarded as a reduced dimensional
map mentioned in the previous step.

\item In the last step, we discuss the approximation of the integrals
defined in coarse-grid system and the coarse-grid system.
%In \cite{eq:nonlinearNLMC_1}, $U^{ms}$ depends on  $\overline{U}$, which
%has a reduced dimension. To compute the integrals, we use periodicity type
%assumption
We seek $\overline{U}$ such that
\[
%\langle  \langle M U^{ms}_t + \nabla \cdot G(x,t, U^{ms}),\psi_i^{(j)} \rangle \rangle
%=
\langle  \langle M  \mathcal{R}_h( \mathcal{R}_H(\overline{U}))_t + \nabla \cdot G(x,t, \mathcal{R}_h( \mathcal{R}_H(\overline{U})),\psi_i^{(j)} )\rangle \rangle_{RVE} = \langle \langle g, \psi_i^{(j)} \rangle \rangle_{RVE},
\]
where $\psi_i^{(j)}$ are standard (e.g., linear test functions).
\end{itemize}

\subsection{Examples}

{\it Example 1.} We consider
\[
-div \kappa(x,\nabla u) = f,
\]
where $\kappa(x,\xi)$ is monotone with respect to $\xi$ (see \cite{leung2019space}).
The algorithm is the following. We seek $\overline{U}^H=\sum_{i,j}\overline{U}_i^{(j)} \overline{\Phi}_i^{j}$, such that
\[
\int_\Omega \kappa(x,\nabla \mathcal{R}_h^K(\mathcal{R}_H(\overline{U})))\cdot \overline{\Phi}_i^{j}dx = \int_\Omega f \overline{\Phi}_i^{j}dx,
\]
where $\mathcal{R}_H(\overline{U})$ is $L^2$ projection of $\overline{U}$
onto $h$-size mesh in RVE,
and $\mathcal{R}_h(\mathcal{R}_H(\overline{U}))$ is
the local RVE solution defined on the fine grid
 with constraints given by $\mathcal{R}_H(\overline{U})$ (which is defined
on $h$-size mesh),
\[
\int_\Omega \kappa(x,\nabla \mathcal{R}_h(\mathcal{R}_H(\overline{U})))\cdot \overline{\Phi}_i^{j}dx\approx \sum_K \omega_K \int_{K_{RVE}} \kappa(x,\nabla \mathcal{R}_h^K(\mathcal{R}_H(\overline{U})))\cdot \overline{\Phi}_i^{j}dx
\]
\[
\int_\Omega f \overline{\Phi}_i^{j}dx\approx  \sum_K \omega_K \int_{K_{RVE}} f \overline{\Phi}_i^{j}dx.
\]
We again note that $\mathcal{R}_h^K$ is local map,
$\mathcal{R}_h^K:V_h(K_{RVE})\rightarrow V(K_{RVE})$,
and $\mathcal{R}_H:V_H\rightarrow \oplus_KV_h(K_{RVE})$ couples different coarse regions.

{\it Example 2.} We consider a simpler example
\[
-div (\kappa(x,u)\nabla u) = f.
\]
In this case, we can consider a linearization (Picard) as
\[
-div (\kappa(x,u^n)\nabla u^{n+1}) = f.
\]
Then, the algorithm is a special case of {\it Example 1}. We seek
$\overline{U}^{n+1}$, which solves the linearized equations. In this
case, we can also define effective permeabilities as we will do in our
numerical examples.

\section{Detailed study of linear case}
\label{sec:detail}
In this section, we will discuss a linear case in detail. We consider the following variational problem: find $u_{\epsilon}\in V$ such that
\begin{align}
a_{\epsilon}(u_{\epsilon},v)= \int_{\Omega} f v,\;\forall v\in V,
\end{align}
where $a_{\epsilon} : V\times V \rightarrow \mathbb{R}$ is a bilinear form, $\epsilon$ is the size of the microscopic scale, $V$ is a Hilbert space which is compactly embedded in $L^2(\Omega)$, $f$ is a source function in $L^2(\Omega)$ and $\Omega$ is the computational domain in $\mathbb{R}^d$. Before introducing our NLMC upscaling method, we will first define the macroscopic quantities for the solution $u_{\epsilon}$. We assume that there is a set of weighted functions, denoted by $\{ \phi_{j}^{h,\epsilon}(y,x) \}$, which can capture the major features of the solution. Under this assumption, we can define a macroscopic system to compute the macroscopic quantities of the solution. In the following subsections, we will construct the upscaled bilinear operators $\tilde{a}_{ij}$ such that the solution $u_j\in V_j$ with $u_j(x)\approx \int_{\Omega}\phi_{j}^{h,\epsilon}u_{\epsilon}(y,x)dy$ satisfies
\begin{align}
\sum _j \tilde{a}_{ij}(u_j,v_i)= \int_{\Omega} f v_i,\;\forall v_i\in V_i.
\end{align}
By using these equations, we can construct a Galerkin method to compute the numerical solution. Let $\mathcal{T}_{H}$ be a partition of $\Omega$. We consider finite element spaces $V_{H,i}\subset V_{i}$. For example, we can take $V_{H,i}$ to be the piecewise linear finite element space. The numerical solution $u_{H,i} \in V_{H,i}$ is then computed by solving the following equations
\begin{align}
\sum _j \tilde{a}_{ij}(u_{H,j},v_i)= \int_{\Omega} f v_i,\;\forall v_i\in V_{H,i}.
\end{align}
To reduce the computational cost of the method, we will compute the approximate bilinear operators $a^{RVE}_{ij}$ where $a^{RVE}_{ij} \approx \tilde{a}_{ij}$ by using our proposed RVE concept.

\subsection{Construction of the upscaled system}
In this section, we will introduce the construction of the upscaled bilinear operator $\tilde{a}_{ij}$. Let $\mathcal{T}_{h}$ be a partition of $\Omega$ where $h$ is the mesh size of $\mathcal{T}_h$. This $h$ is chosen to be the intermediate scale size which is smaller than the RVE size but larger than the microscopic feature size. In the following discussions, we will consider $\mathcal{T}_{h}$ to be the partition 
\[
\mathcal{T}_h = \{K|\; K= \Big(x_0+nh+[-\cfrac{h}{2},\cfrac{h}{2}]\Big)\cap\Omega \text{ for } n \in \mathbb{Z}^d\}
\]
where $x_0$ is a point in $\Omega$. We remark that, in the RVE case, we construct the partition locally in the RVE and $x_0$ can chosen to be the center of the RVE.
Next, we use the function
$\psi_{j}^{h,\epsilon}(x,y)$ to represent the j-th continuum for the local
region $K(x)=\Big(x+[-\cfrac{h}{2},\cfrac{h}{2}]^{d}\Big)\cap\Omega$
with
$$\sum_{j}\omega_{j}^{h,\epsilon}(x)\psi_{j}^{h,\epsilon}(x,\cdot)=\cfrac{(\tilde{\kappa}^{\epsilon})^{-1}}{|K(x)|}I_{K(x)}$$
and
$$\int_{\Omega}\tilde{\kappa}^{\epsilon}(y)\psi_{k}^{h,\epsilon}(x,y)\psi_{j}^{h,\epsilon}(x,y)dy=\delta_{jk}$$
where  $\tilde{\kappa}^{\epsilon}\in L^{\infty}$ is a weight function with $\tilde{\kappa}^{\epsilon}\geq\alpha_{0}>0$.
%with $\tilde{\kappa}^{\epsilon}\in L^{\infty}$.
Next, we define
$J^{h}(x)=\{y\in\Omega,y_i=x_i+n_{i}h,n\in\mathbb{Z}^{d}\}$, where $x_i$ denotes the $i$-th component of $x$.
By solving the local problem, we obtain basis functions $\phi^{h,\epsilon}_{j}(x,\cdot)\in V$
and $\mu_{j}^{h,\epsilon}(x,\cdot)\in V_{aux}^{h}:=\text{span}_{z\in J^{h}(x)}\{\psi_k^{h,\epsilon}(z,\cdot)\}$
as
\begin{align*}
a_{\epsilon}(\phi_{j}^{h,\epsilon}(x,\cdot),v) & =\int_{\Omega}\tilde{\kappa}^{\epsilon}(y)\mu_{j}^{h,\epsilon}(x,y)v(y)dy, \quad\forall v\in V,\\
\int_{\Omega}\tilde{\kappa}^{\epsilon}(y)\phi_{j}^{h,\epsilon}(x,y)\psi_{k}^{h,\epsilon}(z,y)dy & =\delta(x,z)\delta_{jk},\quad\forall z\in J^{h}(x),
\end{align*}
where $a_{\epsilon}$ is a symmetric positive bilinear operator with
$a_{\epsilon}(u,u)\geq\alpha\|u\|_{L^{2}(\Omega)}^{2},\;\forall\epsilon>0$.
For example, $a_{\epsilon}(\phi,v)=\int_{\Omega}\kappa(\epsilon,y)\nabla\phi(y)\cdot\nabla v(y)dy$.

We define $\tilde{u}^{h,\epsilon}(x_{0},y)=\sum_{x\in J(x_{0})}\sum_{j}u_{j}^{h,\epsilon}(x_{0},x)\phi_{j}^{h,\epsilon}(x,y)$
as
\begin{equation}
\label{eq:u_tilde}
a_{\epsilon}(\sum_{x\in J^{h}(x_{0})}\sum_{j}u_{j}^{h,\epsilon}(x_{0,}x)\phi_{j}^{h,\epsilon}(x,\cdot),\phi_{k}^{h,\epsilon}(z,\cdot))=\int_{\Omega}f(y)\phi_{k}^{h,\epsilon}(z,y)\quad\forall z\in J^{h}(x_{0}).
\end{equation}
By the definition of $\phi_{j}^{h,\epsilon}(x,y)$, we have
\[
a_{\epsilon}(\sum_{x\in J^{h}(x_{0})}\sum_{j}u_{j}^{h,\epsilon}(x_{0},x)\phi_{j}^{h,\epsilon}(x,\cdot),v(\cdot))=\int_{\Omega}\tilde{f}^{h,\epsilon}(x_{0},y)v(y)dy, \;\forall v\in V,
\]
where $\tilde{f}_{k}^{h,\epsilon}(x_{0},y)=\sum_{z\in J^{h}(x_{0})}f_{k}(z)\tilde{\kappa}^{\epsilon}(y)\psi_{k}^{h,\epsilon}(z,y)$
and $f_{k}(z)=\int_{\Omega}f(y)\phi_{k}^{h,\epsilon}(z,y)$.

It is clear that for all $x_{0}$, we have $\tilde{u}^{h,\epsilon}(x_{0},\cdot)\in V$
and $\tilde{f}_{k}^{h,\epsilon}(x_{0},\cdot)\in L^{2}(\Omega)$. Using the idea in \cite{chung2018constraintCMAME}, we have
\[
\|\tilde{u}^{h,\epsilon}(x_{0},\cdot)-u^{\epsilon}\|_{a}\leq C^{\epsilon}(h),\quad\;\forall x_{0},
\]
and
\[
\int_{\Omega}u_{j}^{h,\epsilon}(x_{0},x)\psi_{j}^{h,\epsilon}(x,\cdot)=\int_{\Omega}u^{\epsilon}(y)I_{K(x)}(y),\quad\;\forall x\in J^{h}(x_{0}).
\]
To simplify the notation, we will consider a fixed $x_{0}$ and neglect
the index $x_{0}$, for example,
\[
u_{j}^{h,\epsilon}(x)=u_{j}^{h,\epsilon}(x_{0},x).
\]
In addition, we define two different norms $\|\cdot\|_{a,\epsilon}$ and
$\|\cdot\|_{s,\epsilon}$ as
\begin{align*}
\|u\|_{a,\epsilon}^{2} & =a_{\epsilon}(u,u),
\end{align*}
\begin{align*}
\|u\|_{s,\epsilon}^{2} & =\|u\|_{L^{2}(\tilde{\kappa},\Omega)}^{2}=\int_{\Omega}\tilde{\kappa}^{\epsilon}\mu^{2}.
\end{align*}

\begin{defn}
\label{def:operator}
We define a restriction operator $R_{0,k}^{h}:L^{2}(\Omega)\rightarrow l^{2}(J^{h}(x_{0}))$
and two prolongation operators $P_{0,k}^{h}:l^{2}(J^{h}(x_{0}))\rightarrow L^{2}(\Omega)$,
$P_{1,k}^{h}:l^{2}(J^{h}(x_{0}))\rightarrow V$ by
\[
\Big(R_{0,k}^{h}(u)\Big)(z)=\int_{\Omega}\tilde{\kappa}(y)u(y)\psi_{k}^{h,\epsilon}(z,y)dy
\]
\[
P_{0,k}^{h}(v)=\sum_{z\in J^{h}(x_{0})}v(z)\psi_{k}^{h,\epsilon}(z,y)
\]
and
\[
P_{1,k}^{h}(v)=\sum_{z\in J^{h}(x_{0})}v(z)\phi_{k}^{h,\epsilon}(z,y)
\]
Moreover, we define an operator $\Pi_{k}^{h}:L^{2}(\Omega) \rightarrow L^{2}(\Omega)$
such that $\Pi_{k}^{h}=P_{0,k}^{h}\circ R_{0,k}^{h}$.
We can easily check that $\sum_{k}\Pi_{k}^{h}(u)=u$ for all $u\in\text{span}_{z\in J^{h}(x)}\{\psi^{h,\epsilon}_k(z,\cdot)\}$.
\end{defn}

Next, we make the following assumption.

\noindent
{\bf Assumption 1}: For all $\mu\in\text{span}_{z\in J^{h}(x)}\{\psi_k^{h,\epsilon}(z,\cdot)\}$,
there exists $v\in V$ with $\text{supp}\{v\}\subset\text{supp}\{\mu\}$ such that
$$\int_{\Omega}\tilde{\kappa}^{\epsilon}\mu v\geq c^{\epsilon}(h)\Big(\int_{\Omega}\tilde{\kappa}^{\epsilon}\mu^{2}\Big)^{\frac{1}{2}}\Big(a_{\epsilon}(v,v)\Big)^{\frac{1}{2}}$$
and
\[
\Big(\int_{\Omega}\tilde{\kappa}^{\epsilon}\Big|(I-\Pi_{k}^{h})u\Big|^{2}\Big)^{\frac{1}{2}}\leq C^{\epsilon}(h)\Big(a_{\epsilon}(v,v)\Big)^{\frac{1}{2}}
\]
where $C^{\epsilon}$, $c^{\epsilon}$ are monotonic decreasing functions with respect to $\epsilon$. We also assume that
there is a function $\beta:\mathbb{R}\rightarrow\mathbb{R}$ such
that $C^{\epsilon}(\beta(\epsilon))\rightarrow0$ and $\beta(\epsilon)\rightarrow0$.

We next prove the following properties for the operators $R^h_{0,k}, P^h_{0,k}$ and $P^h_{1,k}$.

\begin{lem}
For the operators $R^h_{0,k}, P^h_{0,k}$ and $P^h_{1,k}$ defined in {\bf Definition \ref{def:operator}}, the following hold
\[
\|R_{0,k}^{h}\|_{L^{2}(\tilde{\kappa})}=1,
\]
\[
\|P_{0,k}^{h}(v)\|_{L^{2}(\Omega)}=\|v\|_{l^{2}}, \quad\;\forall v\in l^{2}(J^{h}(x_{0})),
\]
\[
\|P_{1,k}^{h}(v)\|_{a,\epsilon}\leq c^{\epsilon}(h)^{-1}\|v\|_{l^{2}},\quad\;\forall v\in l^{2}(J^{h}(x_0)).
\]
In addition, we have
\[
\|P_{1,k}^{h}\circ R_{0,k}^{h}(u)\|_{a,\epsilon}\leq\|u\|_{a,\epsilon}, \quad\;\forall u\in V.
\]
\end{lem}

\begin{proof}
By the orthogonality of $\psi_{k}^{h,\epsilon}(z,\cdot)$, we obtain
the first two inequalities. For the third inequality, we note that $P_{1,k}^{h}(u)$
satisfies
\begin{align*}
a_{\epsilon}(P_{1,k}^{h}(u),v) & =\sum_{j,z}\mu_{j}(z)\int_{\Omega}\tilde{\kappa}^{\epsilon}\psi_{j}^{h,\epsilon}(z,y)v(y)dy, \quad\forall v\in V, \\
\int_{\Omega}\tilde{\kappa}^{\epsilon}P_{1,k}^{h}(u)\psi_{j}^{h,\epsilon}(z,y)dy & =\delta_{jk}u(z),\quad\forall z\in J^{h}(x).
\end{align*}
Therefore, we have
\begin{align*}
a_{\epsilon}(P_{1,k}^{h}(u),P_{1,k}^{h}(u)) & =\sum_{j,z}\mu_{j}(z)\int_{\Omega}\tilde{\kappa}^{\epsilon}\psi_{j}^{h,\epsilon}(z,y)P_{1,k}^{h}(u)dy\\
 & =\sum_{z}\mu_{k}(z)u(z)\leq\|\mu_{k}\|_{l^{2}}\|u\|_{l^{2}}.
\end{align*}
By {\bf Assumption 1}, there is $v\in V$ such that
$$\sum_{z}\mu_{k}(z)\int_{\Omega}\tilde{\kappa}^{\epsilon}\psi_{k}^{h,\epsilon}(z,y)vdy\geq c^\epsilon(h) \|\sum_{z}\mu_{k}(z)\psi_{k}^{h,\epsilon}(z,\cdot)\|_{s,\epsilon}\|v\|_{a,\epsilon}=c^\epsilon(h)\|\mu_{k}\|_{l^{2}}\|v\|_{a,\epsilon},$$
so, we have 
\begin{align*}
\|\mu_{k}\|_{l^{2}} & \leq\cfrac{\sum_{z}\mu_{k}(z)\int_{\Omega}\tilde{\kappa}^{\epsilon}\psi_{k}^{h,\epsilon}(z,y)vdy}{c^{\epsilon}(h)\|v\|_{a,\epsilon}}\\
 & =\cfrac{a_{\epsilon}(P_{1,k}^{h}(u),v)}{c^{\epsilon}(h)\|v\|_{a,\epsilon}}\leq c^{\epsilon}(h)^{-1}\|P_{1,k}^{h}(u)\|_{a,\epsilon}.
\end{align*}
To obtain the fourth inequality, we note that $P_{1,k}^{h}\circ R_{0,k}^{h}(u)$
is the solution of the following minimization problem
\[
P_{1,k}^{h}\circ R_{0,k}^{h}(u)=\underset{v\in V}{\text{argmin}}\{a_{\epsilon}(v,v)|\;R_{0,k}^{h}(v)=R_{0,k}^{h}(u)\}.
\]
%therefore, we have the fourth inequality.
This completes the proof of the lemma.
\end{proof}

The following lemma gives an error estimate for $\tilde{u}^{h,\epsilon}$ defined in (\ref{eq:u_tilde}).

\begin{lem}
For the function $\tilde{u}^{h,\epsilon}$ defined in (\ref{eq:u_tilde}),
we have
\[
\|\tilde{u}^{h,\epsilon}-u^{\epsilon}\|_{a,\epsilon}=\|\sum_{k}P_{1,k}^{h}\circ R_{0,k}^{h}(u^{\epsilon})-u^{\epsilon}\|_{a,\epsilon}\leq C^{\epsilon}(h)\|(\kappa^{\epsilon})^{-\frac{1}{2}}f\|_{L^{2}(\Omega)}
\]
and
\[
\|\sum_{x\in J^{h}}u_{j}^{h,\epsilon}(x)\psi_{j}^{h,\epsilon}(x,\cdot)-u^{\epsilon}\|_{s,\epsilon}=\|\sum_{k}\Pi_{k}^{h}(u^{\epsilon})-u^{\epsilon}\|_{s,\epsilon}\leq C^{\epsilon}(h)\|u^{\epsilon}\|_{a,\epsilon}.
\]
\end{lem}

\begin{proof}
The first inequality is obtained by the concepts developed in \cite{chung2018constraintCMAME}.
The second inequality is a direct consequence of {\bf Assumption 1}.
\end{proof}

Now, we present our upscaled quantities.

\begin{defn}
We define $\overline{\kappa}_{jk}^{h,\epsilon}$ and $\tilde{\kappa}_{jk}^{h,\epsilon}$
as
\[
\overline{\kappa}_{jk}^{h,\epsilon}(x,z)=a_{\epsilon}(\phi_{j}^{h,\epsilon}(x,\cdot),\phi_{k}^{h,\epsilon}(z,\cdot)), \quad\;\forall x,y\in J^{h},
\]
and
\[
\tilde{\kappa}_{jk}^{h,\epsilon}(x,z)=\sum_{\tilde{x}\in J}\sum_{\tilde{z}\in J}\psi_{j}^{h,\epsilon}(\tilde{x},x)\overline{\kappa}_{jk}^{h,\epsilon}(\tilde{x},\tilde{z})\psi_{k}^{h,\epsilon}(\tilde{z},z).
\]
We remark that for all $u,v\in L^{2}(\Omega)$, we have
\[
\int_{\Omega}\int_{\Omega}v(z)\tilde{\kappa}^{\epsilon}(z)\tilde{\kappa}_{jk}^{h,\epsilon}(x,z)\tilde{\kappa}^{\epsilon}(x)u(x)=\sum_{x,z\in J}R_{0,k}^{h}(u)(x)\overline{\kappa}_{jk}^{h,\epsilon}(x,z)R_{0,j}^{h}(v)(x).
\]
Thus we have
\[
\sum_{j}\int_{\Omega}\int_{\Omega}v(y)\tilde{\kappa}^{\epsilon}(y)\tilde{\kappa}_{jk}^{h,\epsilon}(x,y)\tilde{\kappa}^{\epsilon}(x)\tilde{u}^{h,\epsilon}(x)=\int_{\Omega}\tilde{f}_{k}^{h,\epsilon}(x_{0},y)v(y)dy, \quad\;\forall v\in V.
\]
Using $R_{0,j}^{h}(\tilde{u}^{h,\epsilon}(x))=R_{0,j}^{h}(u^{\epsilon}(x))$, we obtain
\[
\sum_{j}\int_{\Omega}\int_{\Omega}v(y)\tilde{\kappa}^{\epsilon}(y)\tilde{\kappa}_{jk}^{h,\epsilon}(x,y)\tilde{\kappa}^{\epsilon}(x)u^{\epsilon}(x)=\int_{\Omega}\tilde{f}_{k}^{h,\epsilon}(y)v(y)dy, \quad\;\forall v\in V.
\]

We remark that the bilinear operator $\tilde{a}^{h,\epsilon}$
defined as
$$\tilde{a}^{h,\epsilon}(u,v):=\int_{\Omega}\int_{\Omega}v(y)\tilde{\kappa}^{\epsilon}(y)\tilde{\kappa}_{jk}^{h,\epsilon}(x,y)\tilde{\kappa}^{\epsilon}(x)\tilde{u}^{h,\epsilon}(x)$$
is not positive neither in $L^{2}(\Omega)$ nor in $V$. On the other hand, one can easily prove
that the bilinear operator $\tilde{a}^{h,\epsilon}$ is positive in
$V_{aux}^{h}$ and $V^{h}=\text{span}\{\phi_{j}^{h,\epsilon}\}$ where
$\cup_{j=1}^{\infty}V_{aux}^{h_{j}}=L^{2}(\Omega)$, $\cup_{j=1}^{\infty}V^{h_{j}}\subset V$
for any sequence $h_{j}\rightarrow 0$.
\end{defn}

For the next result, we make the following assumptions.

%\textbf{Assumption:}
%\begin{enumerate}

 \noindent
 {\bf Assumption 2}:
 We assume that, there exist $\tilde{V}$
such that for all $f\in L^{2}(\Omega)$, the solution $\Pi_{k}^{\beta(\epsilon)}u^{\epsilon}\rightarrow u_{k}^{0}$
in $L^{2}(\Omega)$ and $u_{k}^{0}\in\tilde{V}_{k}$ satisfies
\[
\sum_{k}a_{j,k}^{0}(u_{k}^{0},v_{j})=(f_{j}^{0},v_{j}), \quad\;\forall v\in\tilde{V}_{j},
\]
for some $f^{0}_j$.

\noindent
{\bf Assumption 3}:
We assume that, there exist $\tilde{V}$
such that for all $f\in L^{2}(\Omega)$, $\Pi_{k}^{\beta(\epsilon)}\tilde{u}^{\beta(\epsilon),\epsilon}\rightarrow\tilde{u}_{k}^{0}$
in $L^{2}(\Omega)$ and $\tilde{u}_{k}^{0}\in\tilde{V}_{k}$ satisfies
\[
\sum_{k}\tilde{a}_{j,k}^{0}(\tilde{u}_{k}^{0},v_{j})=(\tilde{f}_{j}^{0},v), \quad\;\forall v\in\tilde{V}_{j},
\]
for some $\tilde{f}_j^{0}$.

\begin{lem}
{\bf Assumption 2} holds if and only if {\bf Assumption 3} holds.
\end{lem}

\begin{proof}
We assume {\bf Assumption 2} holds. First, we let $\tilde{u}_{k}^{0}=u_{k}^{0}$
and obtain
\begin{align*}
\|\Pi_{k}^{\beta(\epsilon)}(\tilde{u}^{h,\epsilon})-\tilde{u}_{k}^{0}\|_{L^{2}(\Omega)} & =\|\Pi_{k}^{\beta(\epsilon)}(\tilde{u}^{h,\epsilon})-u_{k}^{0}\|_{L^{2}(\Omega)}\\
 & \leq\|\Pi_{k}^{\beta(\epsilon)}\tilde{u}^{h,\epsilon}-\Pi_{k}^{\beta(\epsilon)}u^{\epsilon}\|_{L^{2}(\Omega)}+\|\tilde{u}_{k}^{0}-\Pi_{k}^{\beta(\epsilon)}u^{\epsilon}\|_{L^{2}(\Omega)}\\
 & =\|\tilde{u}_{k}^{0}-\Pi_{k}^{\beta(\epsilon)}u^{\epsilon}\|_{L^{2}(\Omega)}\rightarrow0.
\end{align*}
We consider $\tilde{a}_{j,k}^{0}(u,v)=a_{j,k}^{0}(u,v)\;\forall u\in\tilde{V}_{k},v\in\tilde{V}_{j}$
and $\tilde{f}_{j}^{0}=f_{j}^{0}$. Therefore, we obtain
\begin{align*}
\sum_{k}\tilde{a}_{j,k}^{0}(\tilde{u}_{k}^{0},v_{j}) & =\sum_{k}a_{j,k}^{0}(u_{k}^{0},v_{j})=(f_{j}^{0},v_{j})\\
 & =(\tilde{f}_{j}^{0},v_{j})
\end{align*}
for all $v\in\tilde{V}_{j}$.
The statement that {\bf Assumption 3}  implies {\bf Assumption 2} can be proved by a similar argument.
\end{proof}

Our final error estimate is based on the following assumption.

%\textbf{Assumption:}

\noindent
{\bf Assumption 4}:
We assume that there exists a bilinear form $\tilde{a}_{jk}:\tilde{V}_{k}\times\tilde{V}_{j}\rightarrow\mathbb{R}$ such
that
$$\cfrac{|\tilde{a}_{jk}(u,v)-\tilde{a}_{jk}^{\beta_{0}(\epsilon),\epsilon}(u,v)|}{\|u\|_{\tilde{V}_{k}}\|v\|_{\tilde{V}_{j}}}\rightarrow0$$
where we define
$$\tilde{a}_{jk}^{h,\epsilon}(u,v):=\int_{\Omega}\int_{\Omega}v(y)\tilde{\kappa}^{\epsilon}(y)\tilde{\kappa}_{jk}^{h,\epsilon}(x,y)\tilde{\kappa}^{\epsilon}(x)u(x).$$
Furthermore, we assume there are constants $c_0, c_1$ and $C_1$ such that
$$c_{0}\sum_{k}\|u_{k}\|_{L^{2}}^{2}\leq c_{1}\sum_{k}\|u_{k}\|_{\tilde{V}_{k}}^{2}\leq\sum_{j,k}\tilde{a}_{jk}(u_{k},u_{j}),$$
and $$\sum_{j,k}\tilde{a}_{jk}(u_{k},v_{j})\leq C_{1}\Big(\sum_{k}\|u_{k}\|_{\tilde{V}_{k}}^{2}\Big)^{\frac{1}{2}}\Big(\sum_{k}\|u_{k}\|_{\tilde{V}_{k}}^{2}\Big)^{\frac{1}{2}}.$$
We also assume that
$$\tilde{f}_{j}^{\beta_{0}(\epsilon),\epsilon}=\int_{\Omega}f(y)\sum_{z\in J^{h}}\phi_{k}^{h,\epsilon}(z,y)\tilde{\kappa}^{\epsilon}(x)\psi_{k}^{h,\epsilon}(z,x)\rightarrow\tilde{f}_{j}.$$

We state our main result.

\begin{theorem}
\label{thm}
If {\bf Assumption 4} holds, we have $\Pi_{k}^{\beta_{0}(\epsilon)}\tilde{u}^{\beta_{0}(\epsilon),\epsilon}\rightarrow\tilde{u}_{k}$
where $\tilde{u}_{k}\in\tilde{V}_{k}$ satisfies
\[
\sum_{k}\tilde{a}_{j,k}(\tilde{u}_{k},v_{j})=(\tilde{f}_{j},v_{j}).
\]
\end{theorem}

\begin{proof}
First, we have
\[
\sum_{k}\tilde{a}_{jk}^{\beta_{0}(\epsilon),\epsilon}(\Pi_{k}^{\beta_{0}(\epsilon)}\tilde{u}^{\beta_{0}(\epsilon),\epsilon},v_{j})=(\tilde{f}_{j}^{\beta_{0}(\epsilon),\epsilon},v_{j})
\]
and, using {\bf Assumption 4},
\begin{align*}
c_{1}\sum_{k}\|\Pi_{k}^{\beta_{0}(\epsilon)}\tilde{u}^{\beta_{0}(\epsilon),\epsilon}\|_{\tilde{V}_{k}}^{2} & \leq\sum_{j,k}\tilde{a}_{jk}(\Pi_{k}^{\beta_{0}(\epsilon)}\tilde{u}^{\beta_{0}(\epsilon),\epsilon},\Pi_{j}^{\beta_{0}(\epsilon)}\tilde{u}^{\beta_{0}(\epsilon),\epsilon})\\
 & \leq\sum_{j} \Big\{ |\sum_{k}\tilde{a}_{jk}(\Pi_{k}^{\beta_{0}(\epsilon)}\tilde{u}^{\beta_{0}(\epsilon),\epsilon},\Pi_{j}^{\beta_{0}(\epsilon)}\tilde{u}^{\beta_{0}(\epsilon),\epsilon})-\sum_{k}\tilde{a}_{jk}^{\beta_{0}(\epsilon),\epsilon}(\Pi_{k}^{\beta_{0}(\epsilon)}\tilde{u}^{\beta_{0}(\epsilon),\epsilon},\Pi_{j}^{\beta_{0}(\epsilon)}\tilde{u}^{\beta_{0}(\epsilon),\epsilon})| \\
 &\quad\quad+|(\tilde{f}_{j}^{\beta_{0}(\epsilon),\epsilon},\Pi_{j}^{\beta_{0}(\epsilon)}\tilde{u}^{\beta_{0}(\epsilon),\epsilon})| \Big\}.
\end{align*}
When $\epsilon$ is small enough, we have
\[
\sum_{j}|\sum_{k}\tilde{a}_{jk}(\Pi_{k}^{\beta_{0}(\epsilon)}\tilde{u}^{\beta_{0}(\epsilon),\epsilon},\Pi_{j}^{\beta_{0}(\epsilon)}\tilde{u}^{\beta_{0}(\epsilon),\epsilon})-\sum_{k}\tilde{a}_{jk}^{\beta_{0}(\epsilon),\epsilon}(\Pi_{k}^{\beta_{0}(\epsilon)}\tilde{u}^{\beta_{0}(\epsilon),\epsilon},\Pi_{j}^{\beta_{0}(\epsilon)}\tilde{u}^{\beta_{0}(\epsilon),\epsilon})|\leq\cfrac{c_{1}}{2}\|u\|_{\tilde{V}_{k}}\|v\|_{\tilde{V}_{j}}
\]
and
\[
\sum_{k}\|\Pi_{k}^{\beta_{0}(\epsilon)}\tilde{u}^{\beta_{0}(\epsilon),\epsilon}\|_{\tilde{V}_{k}}^{2}\leq C\sum_{j}\|\tilde{f}_{j}^{\beta_{0}(\epsilon)}\|_{L^{2}(\Omega)}^{2}\leq C\sum_{j}\|\tilde{f}_{j}\|_{L^{2}(\Omega)}^{2}.
\]
Let $\eta_{j}=\tilde{u}_{j}-\Pi_{j}^{\beta_{0}(\epsilon)}\tilde{u}^{\beta_{0}(\epsilon),\epsilon}$. Then we have
\begin{align*}
c_{1}\sum_{k}\|\eta_{k}\|_{\tilde{V}_{k}}^{2} & \leq\sum_{k,j}\tilde{a}_{jk}(\tilde{u}_{k}-\Pi_{k}^{\beta_{0}(\epsilon)}\tilde{u}^{\beta_{0}(\epsilon),\epsilon},\eta_{j})\\
 & \leq|\sum_{k,j}\tilde{a}_{jk}(\tilde{u}_{k},\eta_{j})-\tilde{a}_{jk}^{\beta_{0}(\epsilon),\epsilon}(\Pi_{k}^{\beta_{0}(\epsilon)}\tilde{u}^{\beta_{0}(\epsilon),\epsilon},\eta_{j})|+|\sum_{k,j}\tilde{a}_{jk}(\Pi_{k}^{\beta_{0}(\epsilon)}\tilde{u}^{\beta_{0}(\epsilon),\epsilon},\eta_{j})-\tilde{a}_{jk}^{\beta_{0}(\epsilon),\epsilon}(\Pi_{k}^{\beta_{0}(\epsilon)}\tilde{u}^{\beta_{0}(\epsilon),\epsilon},\eta_{j})|\\
 & =|\sum_{j}(\tilde{f}_{j}-\tilde{f}_{j}^{\beta_{0}(\epsilon),\epsilon},\eta_{j})|+|\sum_{k,j}\tilde{a}_{jk}(\Pi_{k}^{\beta_{0}(\epsilon)}\tilde{u}^{\beta_{0}(\epsilon),\epsilon},\eta_{j})-\tilde{a}_{jk}^{\beta_{0}(\epsilon),\epsilon}(\Pi_{k}^{\beta_{0}(\epsilon)}\tilde{u}^{\beta_{0}(\epsilon),\epsilon},\eta_{j})|
\end{align*}
and
\[
\Big(\sum_{k}\|\eta_{k}\|_{\tilde{V}_{k}}^{2}\Big)^{\frac{1}{2}}\leq\cfrac{|\sum_{j}(\tilde{f}_{j}-\tilde{f}_{j}^{\beta_{0}(\epsilon),\epsilon},\eta_{j})|+|\sum_{k,j}\tilde{a}_{jk}(\Pi_{k}^{\beta_{0}(\epsilon)}\tilde{u}^{\beta_{0}(\epsilon),\epsilon},\eta_{j})-\tilde{a}_{jk}^{\beta_{0}(\epsilon),\epsilon}(\Pi_{k}^{\beta_{0}(\epsilon)}\tilde{u}^{\beta_{0}(\epsilon),\epsilon},\eta_{j})|}{\Big(\sum_{k}\|\eta_{k}\|_{\tilde{V}_{k}}^{2}\Big)^{\frac{1}{2}}\Big(\sum_{j}\|\tilde{f}_{j}\|_{L^{2}(\Omega)}^{2}\Big)^{\frac{1}{2}}}\rightarrow0.
\]
This completes the proof.
\end{proof}

\subsection{Two examples}

In this section, we present two examples.
By the decaying property of the function $\phi_{j}^{h,\epsilon}$, we have $\kappa_{j,k}(x,z)=0\;\text{if}$
$x\neq z$.

\noindent
{\bf Case 1}: We take $\kappa(\epsilon,x)=\kappa(\cfrac{x}{\epsilon})$, $\psi_{1}^{h,\epsilon}(x,y)=\cfrac{1}{|K(x)|^{\frac{1}{2}}}I_{K(x)}(y)$ and
$\tilde{\kappa}^{\epsilon}=1$. Then we have

\[
\tilde{u}_{1}^{\epsilon^{\frac{1}{2}},\epsilon}\rightarrow u_{0}^{*}
\]
 where $\int_{\Omega}\kappa^{*}\nabla u_{0}^{*}\cdot\nabla v dx =\int_{\Omega}f v$ and
we deduce from Theorem~\ref{thm} that
\[
\tilde{a}_{jk}^{\epsilon^{\frac{1}{2}},\epsilon}(u,v)\rightarrow\int_{\Omega}\kappa^{*}\nabla u\cdot\nabla vdx, \quad \text{as } \epsilon \rightarrow 0.
\]

\noindent
{\bf Case 2}: We take $\kappa(\epsilon,x)=\kappa_{0}(x)+\kappa_{1}\epsilon^{-1}I_{\Gamma(\epsilon)}$ where
$\Gamma(\epsilon)=\{x\in\Omega \; |\; d(x,\Gamma)<\epsilon\}$ and $\Gamma$ is a fracture in the domain $\Omega$, and
$$\psi_{1}^{h,\epsilon}(x,y)=\cfrac{1}{|K(x)\backslash F(\epsilon)|^{\frac{1}{2}}}I_{K(x)\backslash F(\epsilon)}(y), \quad
\psi_{2}^{h,\epsilon}(x,y)=\cfrac{1}{|K(x)\cap F(\epsilon)|^{\frac{1}{2}}}I_{K(x)\cap F(\epsilon)}(y).$$
If $u^{\epsilon}\rightarrow u^{0}$ in $H^{1}(\Omega)$ and $u^{\epsilon}|_{\Gamma}\rightarrow u_{\Gamma}^{0}$
in $H^{1}(\Gamma)$, we again deduce from Theorem~\ref{thm} that 
\begin{align*}
 & \tilde{a}_{11}^{\epsilon,\epsilon}(u,v)
\rightarrow  \int_{\Omega}\kappa_{0}\nabla u\cdot\nabla vdx+h^{-1}\int_{\Gamma}Quv+O(h), \\
& \tilde{a}_{12}^{\epsilon,\epsilon}(u,v)
\rightarrow  -h^{-1}\int_{\Gamma}Quv+O(h), \\
& \tilde{a}_{21}^{\epsilon,\epsilon}(u,v)\rightarrow-h^{-1}\int_{\Gamma}Quv+O(h), \\
& \tilde{a}_{22}^{\epsilon,\epsilon}(u,v)\rightarrow\int_{\Gamma}\kappa_{1}\nabla_{\Gamma}u\cdot\nabla_{\Gamma}v+h^{-1}\int_{\Gamma}Quv+O(h).
\end{align*}
Hence we obtain the following upscale system
\begin{align*}
\int_{\Omega}\kappa_{0}^{*}\nabla u_{1}\cdot\nabla v+h^{-1}\int_{\Gamma}Q(u_{1}-u_{2})v+O(h) & =\int_{\Omega}\tilde{f}_{1}v, \quad \;\forall v\in H^{1}(\Omega),\\
\int_{\Gamma}\kappa_{0}^{*}\nabla_{\Gamma}u_{2}\cdot\nabla_{\Gamma}v+h^{-1}\int_{\Gamma}Q(u_{2}-u_{1})v+O(h) & =\int_{\Gamma}\tilde{f}_{2}v, \quad\;\forall v\in H^{1}(\Gamma).
\end{align*}
We remark that the operator $Q$ is defined in the limit above.

\subsection{RVE approximation}
In this section, we will discuss using a RVE concept to approximate the bilinear operator $\tilde{a}_{jk}^{h,\epsilon}$.
Since $\tilde{a}_{jk}^{h,\epsilon}(u,v):=\int_{\Omega}\int_{\Omega}v(y)\tilde{\kappa}^{\epsilon}(y)\tilde{\kappa}_{jk}^{h,\epsilon}(x,y)\tilde{\kappa}^{\epsilon}(x)u(x)$,
we can approximate $\tilde{a}_{jk}^{h,\epsilon}$ by using  a suitable quadrature rule for the integral. For each coarse grid element $K\in\mathcal{T}_H$, we consider there is a set of RVEs, $K_{REV,k}$. We will approximate $\tilde{a}_{jk}^{h,\epsilon}$ by $\tilde{a}_{jk}^{RVE}$ such that
\begin{equation}
\tilde{a}_{jk}^{RVE} = \sum_k \omega_{RVE,k}\int_{K_{REV,k}}\int_{K_{REV,k}}v(y)\tilde{\kappa}^{\epsilon}(y)\tilde{\kappa}_{jk}^{h,\epsilon}(x,y)\tilde{\kappa}^{\epsilon}(x)u(x).
\end{equation}

\section{Numerical Results}
\label{sec:num}

In this section, we will consider a numerical example to demonstrate
the performance of the method. The computational domain $\Omega$
is defined as $\Omega=[0,1]^{2}$. The medium parameter $\kappa_{\epsilon}$
is defined as
\[
\kappa_{\epsilon}=\begin{cases}
\cfrac{3}{10\epsilon} & \text{if }x\in\Gamma_{1}^{\epsilon}\backslash\Gamma_{2}^{\epsilon}\\
\cfrac{1}{\epsilon} & \text{if }x\in\Gamma_{2}^{\epsilon}\cup\Gamma_{3}^{\epsilon}\\
\cfrac{7}{10\epsilon} & \text{if }x\in\Gamma_{4}^{\epsilon}\\
1 & \text{if }x\in\Omega\backslash(\Gamma_{1}^{\epsilon}\cup\Gamma_{2}^{\epsilon}\cup\Gamma_{3}^{\epsilon}\cup\Gamma_{4}^{\epsilon})
\end{cases}
\]
where
\begin{align*}
& \Gamma_{1}^{\epsilon}=\{x|\;\Big|x_{1}-\cfrac{1}{2}\Big|<\epsilon,\;\cfrac{1}{4}<x_{2}<\cfrac{7}{8}\}, \\
& \Gamma_{2}^{\epsilon}=\{x|\;\Big|\cfrac{x_{1}+x_{2}-1}{\sqrt{2}}\Big|<\epsilon,\;-\cfrac{1}{4}<x_{1}-x_{2}<\cfrac{3}{4}\}, \\
& \Gamma_{3}^{\epsilon}=\{x|\;\Big|\cfrac{2x_{1}-3x_{2}+0.2}{\sqrt{13}}\Big|<\epsilon,\;\cfrac{7}{8}<\cfrac{3x_{1}+2x_{2}}{\sqrt{13}}<\cfrac{9}{8}\}, \\
& \Gamma_{4}^{\epsilon}=\{x|\;\sqrt{2}\Big|(x_{1}-\cfrac{1}{2})^{2}+(x_{2}-\cfrac{1}{2})^{2}-\cfrac{1}{8}\Big|<\epsilon,\;x_{1}<\cfrac{1}{2},\;x_{2}<\cfrac{3}{4}\}.
\end{align*}
The choice of the RVE location is illustrated in Figure~\ref{fig:RVE}.
The source term $f$ is defined as
\[
f(x) = e^{-40((x_1-\frac{9}{10})^2+(x_2-\frac{1}{10})^2)}.
\]
In Figure~\ref{fig:result1},
%\marginpar{Tat, explain the meaning of the colors in the figure}, 
we present the computational results.
In the first figure (left plot in Figure~\ref{fig:result1}), we present the reference solution. In the second figure (middle plot in Figure~\ref{fig:result1}),
we present the matrix part of the upscaled solution. In the third figure (right plot in Figure~\ref{fig:result1}),
we present the channel part of the upscaled solution. From these results, we observe that our proposed upscaling method
is able to produce accurate upscaled solutions.

\begin{figure}[ht]
\centering

\includegraphics[width=12cm,height=8cm]{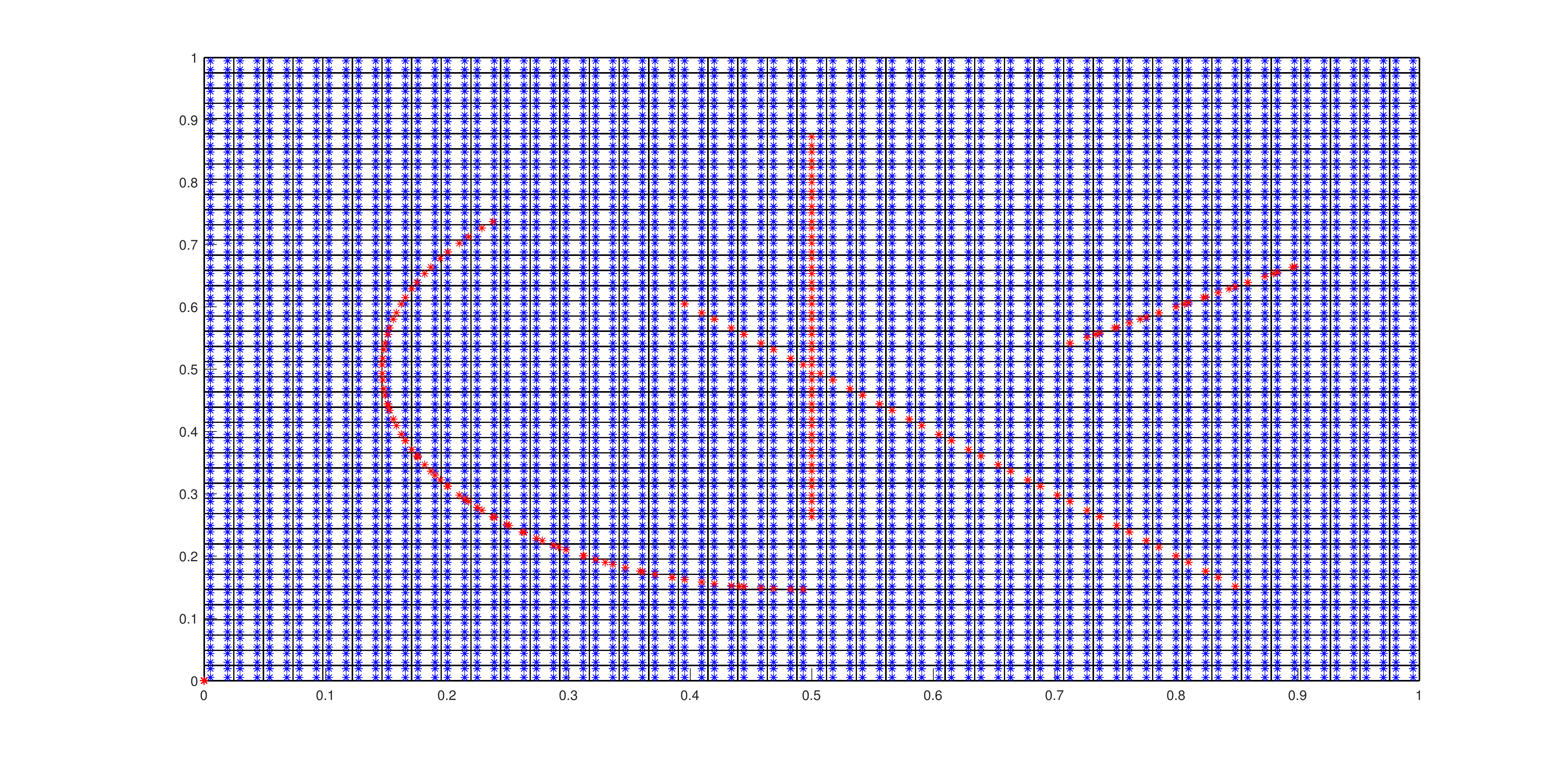}

\caption{The computational domain $\Omega$, the medium parameter $\kappa_{\epsilon}$ and the RVE points.}
\label{fig:RVE}
\end{figure}

\begin{figure}[ht]
\centering

\includegraphics[scale=0.3]{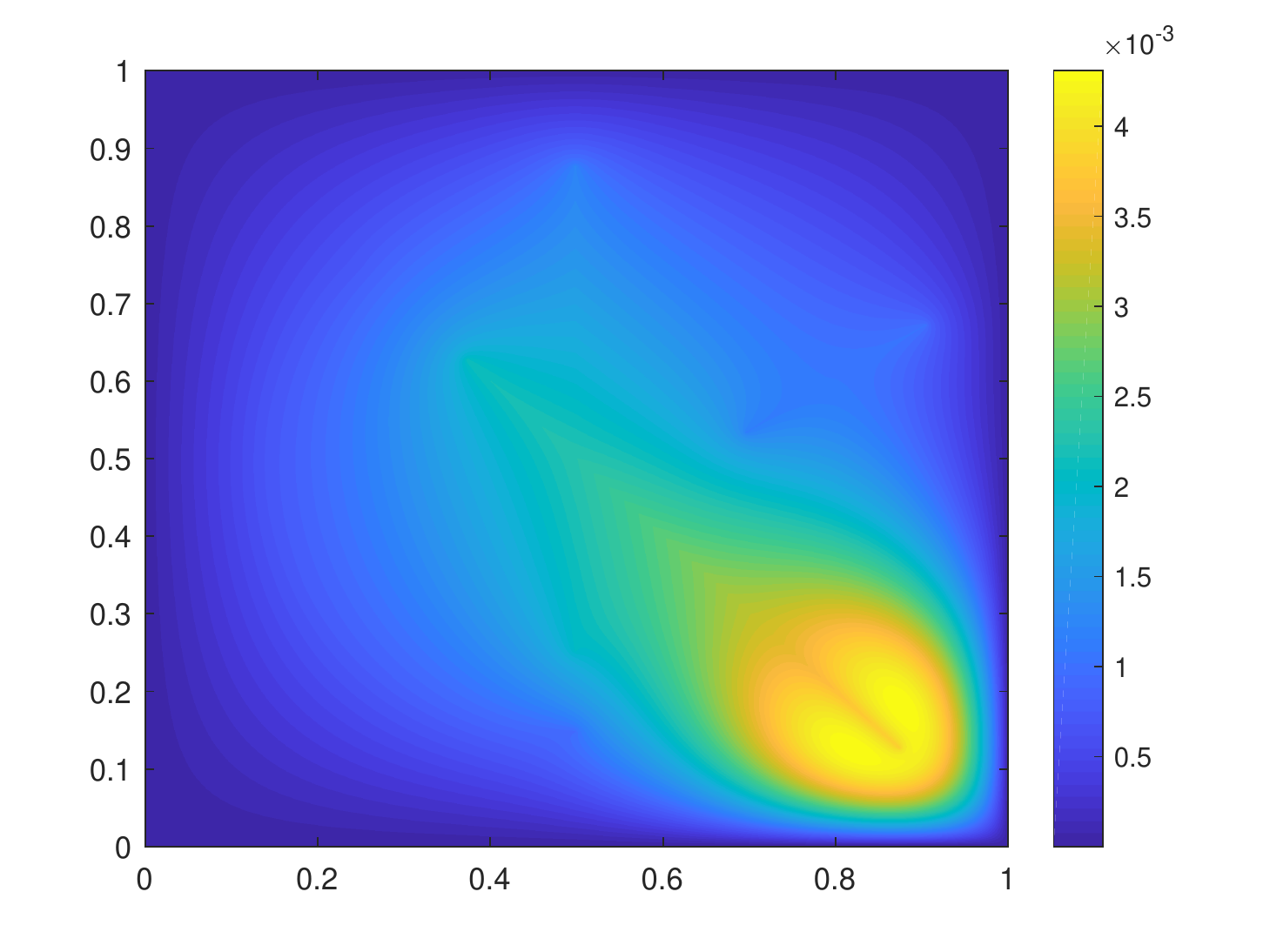} \includegraphics[scale=0.3]{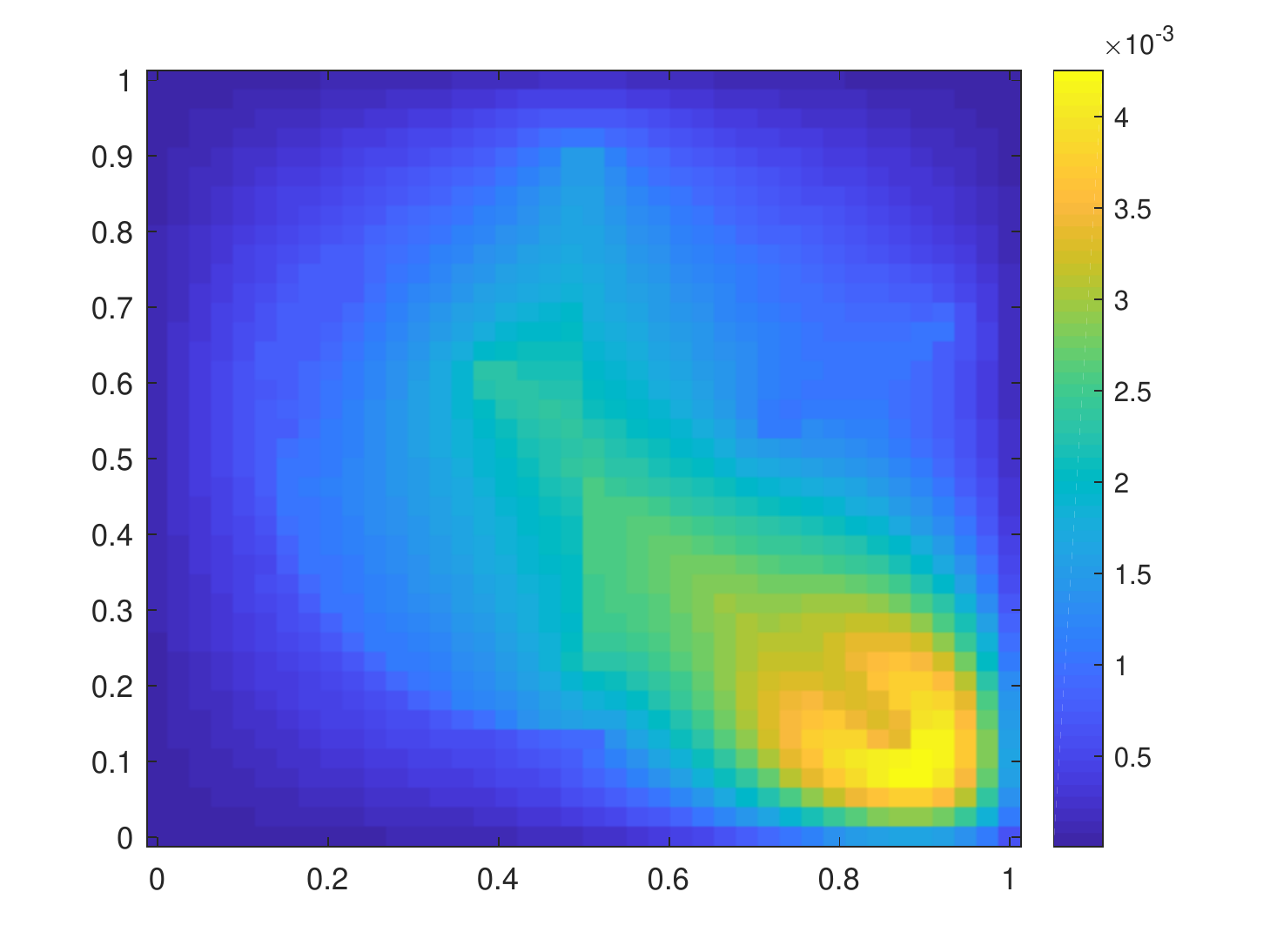}
\includegraphics[scale=0.3]{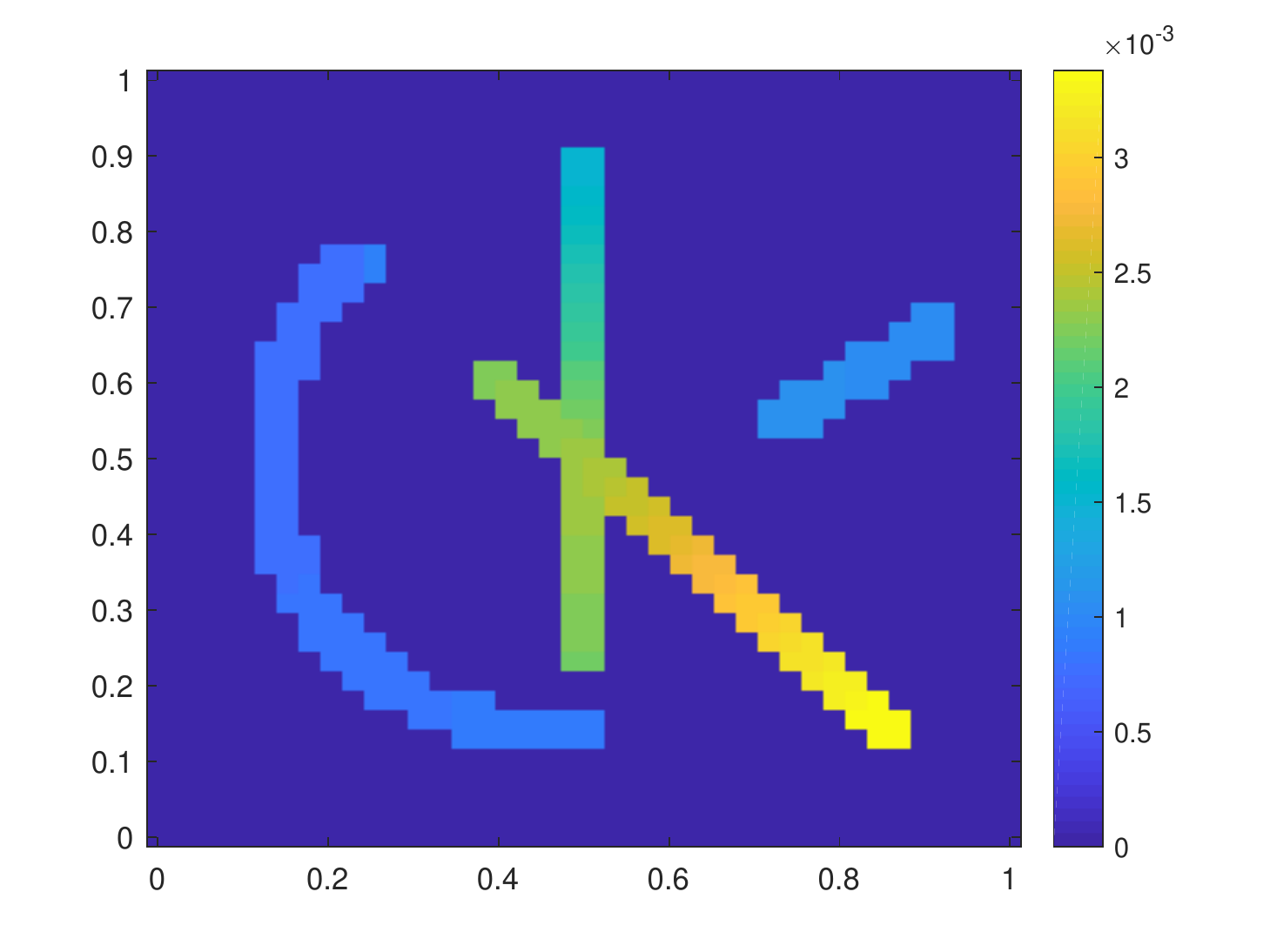}

\caption{Computational results for the first example. Left: reference solution, Middle: upscaled solution (matrix),
Right: upscaled solution (channel).}
\label{fig:result1}
\end{figure}

In the next example, we will consider a time dependent case. We assume $u_1$ and $u_2$ satisfy the following equations
\begin{align}
\Big(\Big(\Pi^{h}_{i}(u_i)(\cdot,t)\Big)_t,\Pi^{h}_{i}(v_i)\Big) + \sum _j \tilde{a}_{ij}(u_j,v_i)= \int_{\Omega} f v_i,\;\forall v_i\in V_i,\;\forall t\in (0,T].
\end{align}
The computational domain $\Omega$
is defined as $\Omega=[0,1]^{2}$. The medium parameter $\kappa_{\epsilon}$
is defined as
\[
\kappa_{\epsilon}=\begin{cases}
10 & \text{if } \Big|\sin \Big(\cfrac{\pi ((1-x_2)x_2+x_1)}{\epsilon}\Big)\sin \Big(\cfrac{\pi(x_2+x^2_1)}{\epsilon}\Big)\Big|<0.2,\\
\cfrac{\epsilon}{10000} & \text{otherwise.}
\end{cases}
\]
The source term $f$ is chosen to be the same as the previous example.
In Figure~\ref{fig:RVE_time}, we illustrate the medium parameter $\kappa_{\epsilon}$ and the RVE points.
In the top figure, we present the computational domain, coarse grid and the RVE points.
In the bottom left figure, we present the medium parameter $\kappa_{\epsilon}$ around the point $(0.0717,0.717)$.
In the bottom middle and bottom right figures, we present the medium parameter $\kappa_{\epsilon}$ around the points
$(0.5262,0.5262)$ and $(0.9808,0.717)$ respectively.

\begin{figure}[ht]
\centering
\includegraphics[width=12cm,height=8cm]{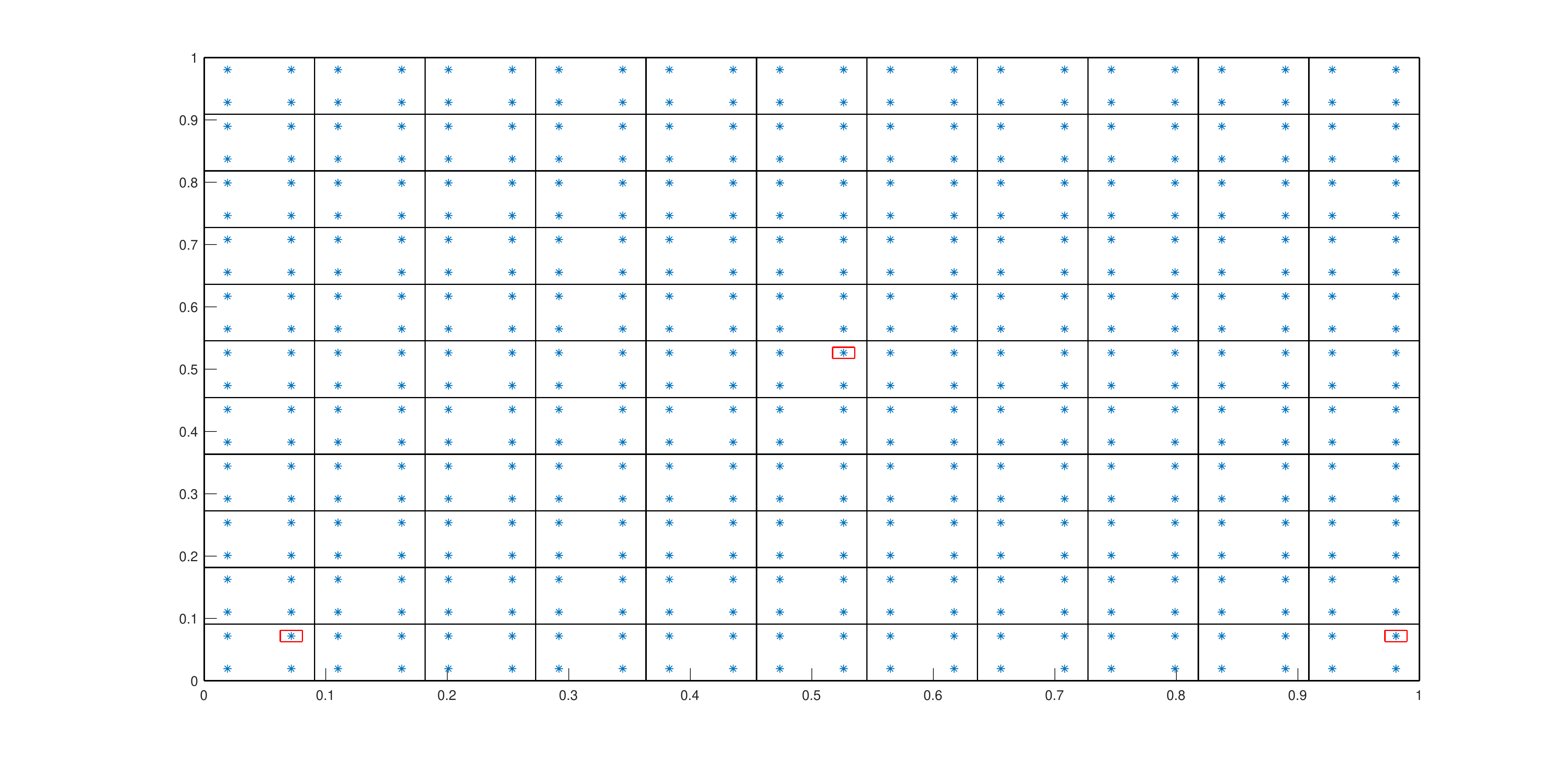}

\includegraphics[scale=0.3]{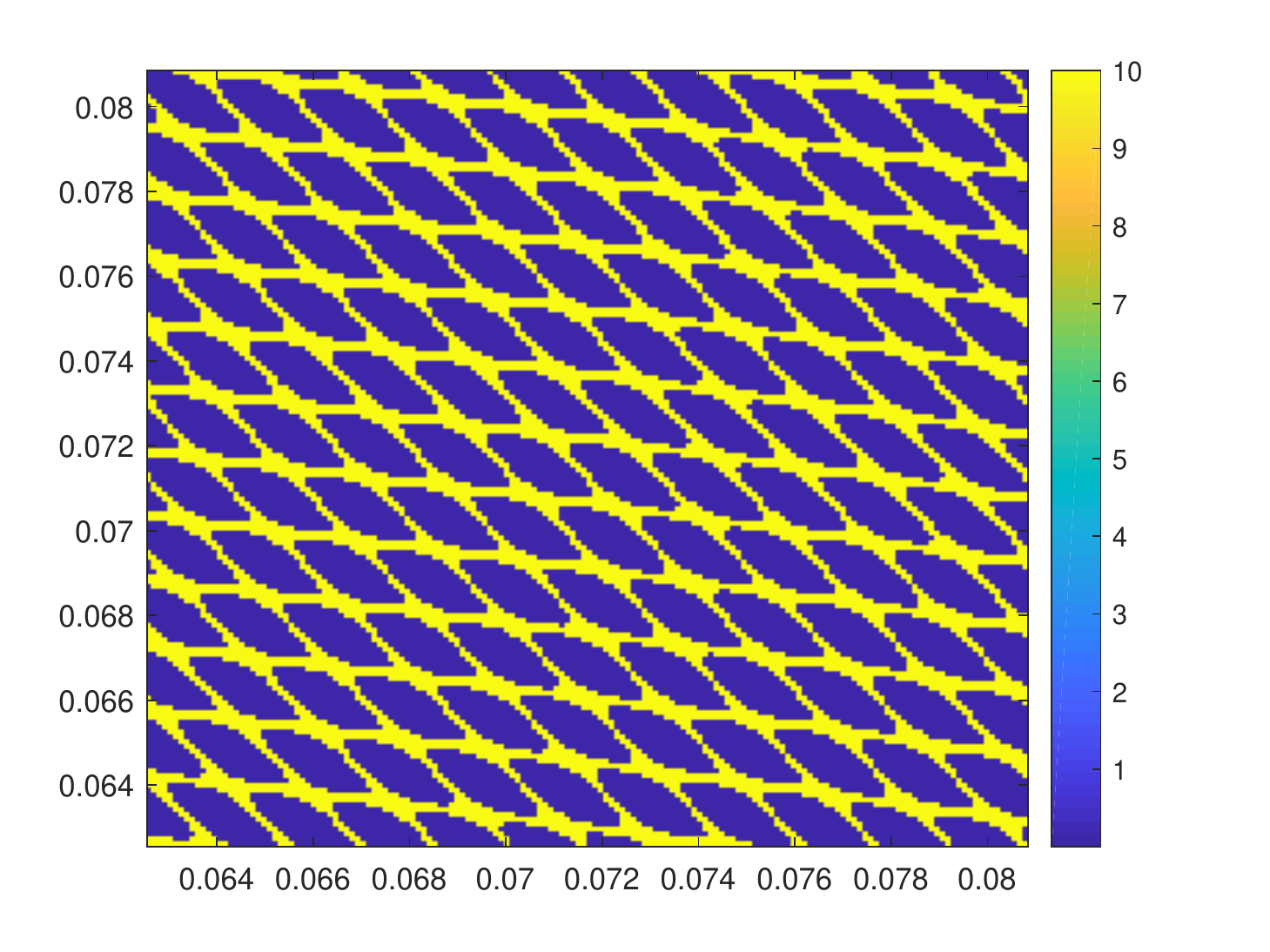} \includegraphics[scale=0.3]{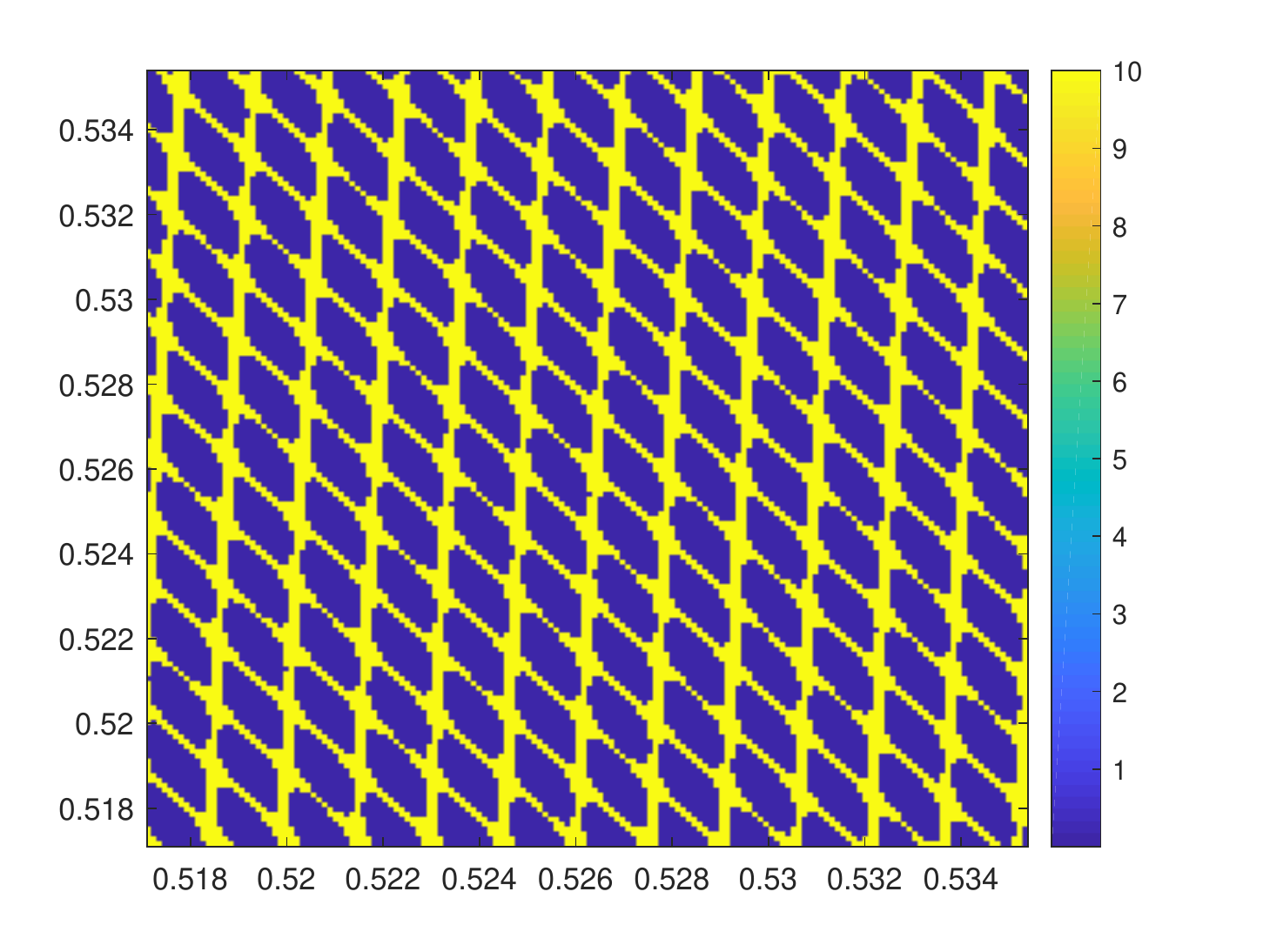} \includegraphics[scale=0.3]{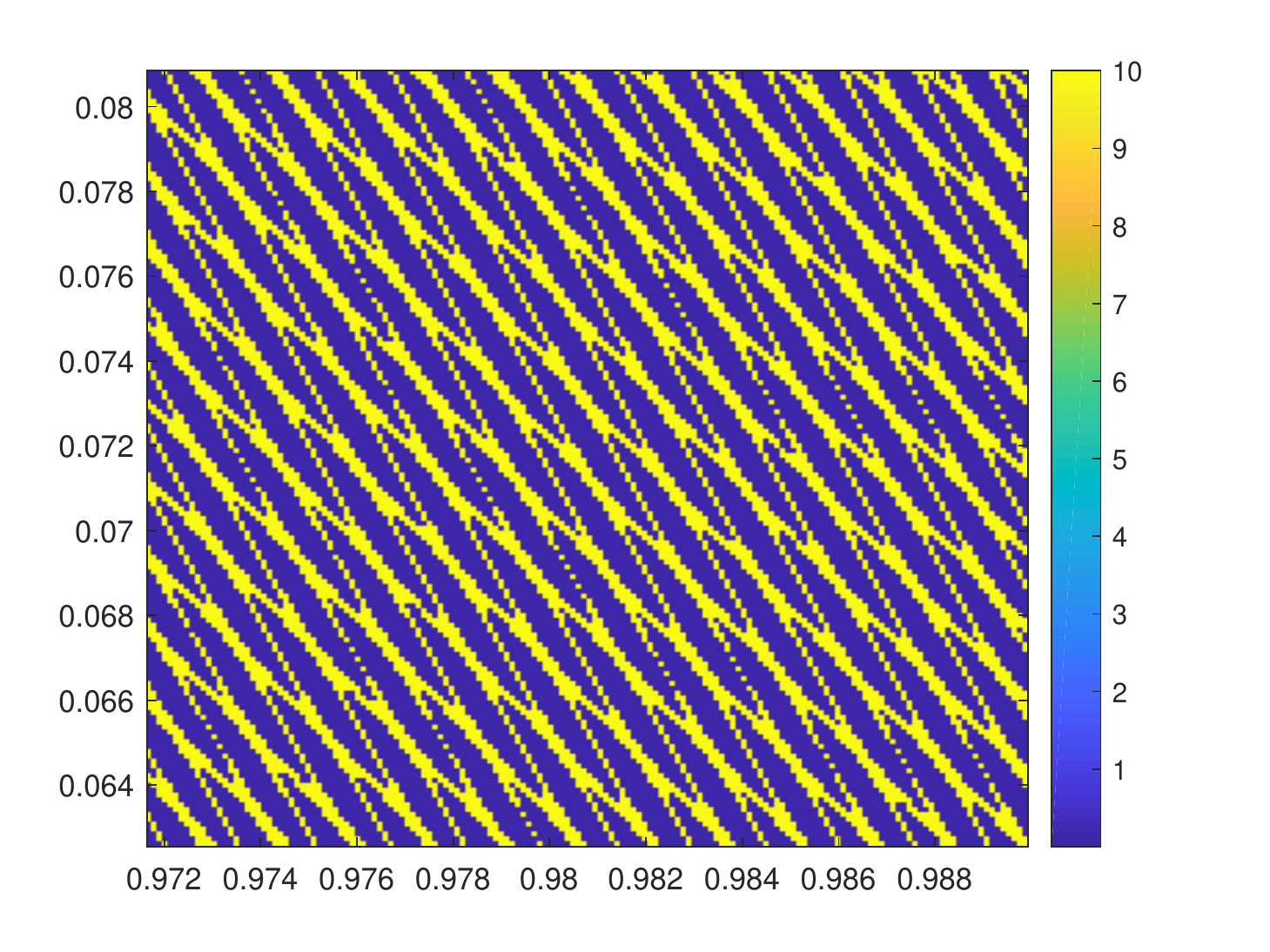}
\caption{The RVE points and the medium parameter $\kappa_{\epsilon}$ for the second example. Top: RVE points in the domain. Bottom Left: $\kappa_{\epsilon}$  around $(0.0717,0.717)$. Bottom Middle: $\kappa_{\epsilon}$  around $(0.5262,0.5262)$. Bottom Left: $\kappa_{\epsilon}$  around $(0.9808,0.717)$.}
\label{fig:RVE_time}
\end{figure}

In Figure~\ref{fig:results}, we present the computational results for the second example.
In the three figures on the top, we present the matrix part of the upscaled solutions at the times $T=0.005, T=0.01$ and $T=0.02$
respectively.
In the three figures at the bottom, we present the channel part of the upscaled solutions at the times $T=0.005, T=0.01$ and $T=0.02$
respectively.

\begin{figure}[ht]
\centering

\includegraphics[scale=0.3]{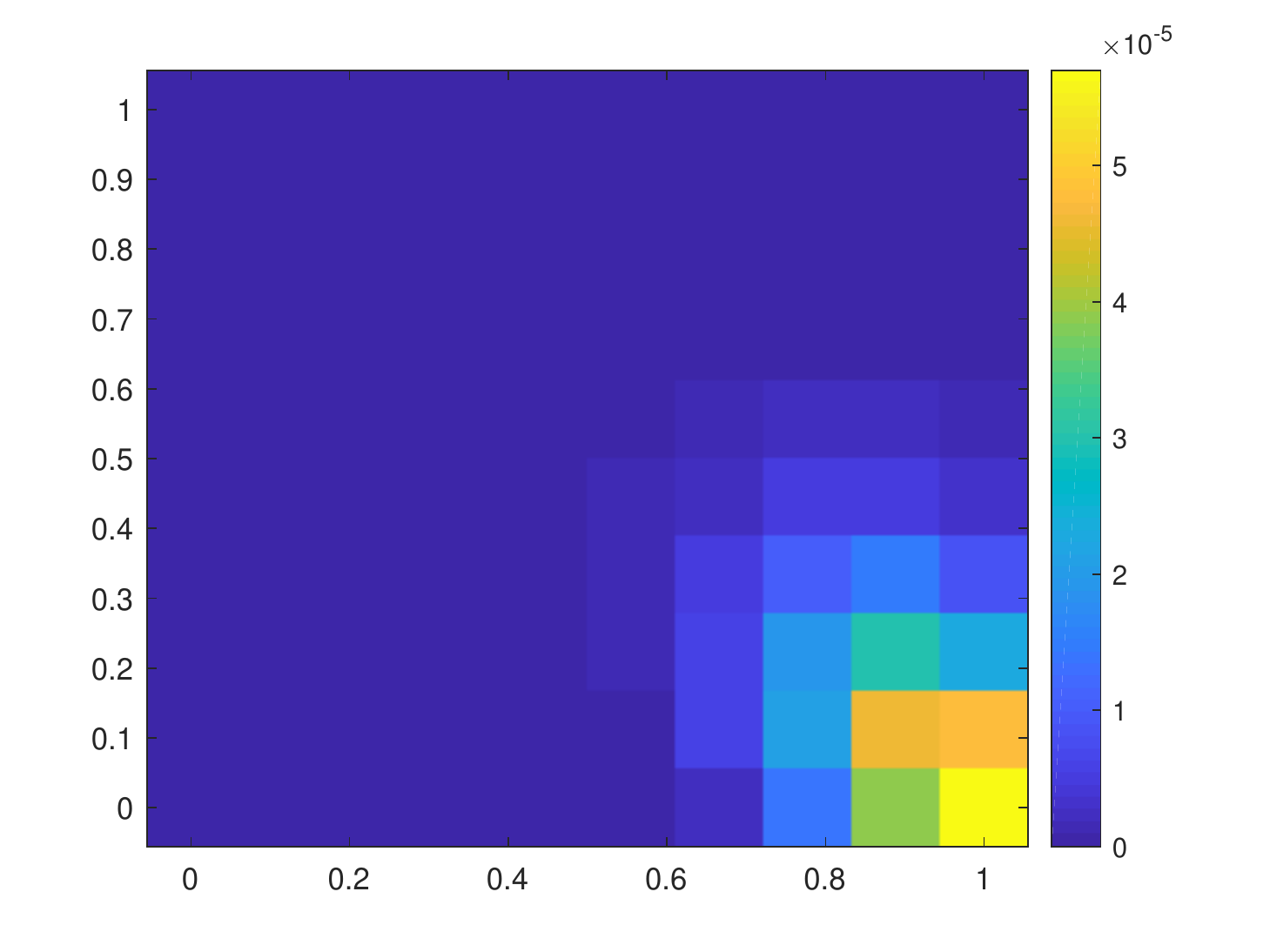} \includegraphics[scale=0.3]{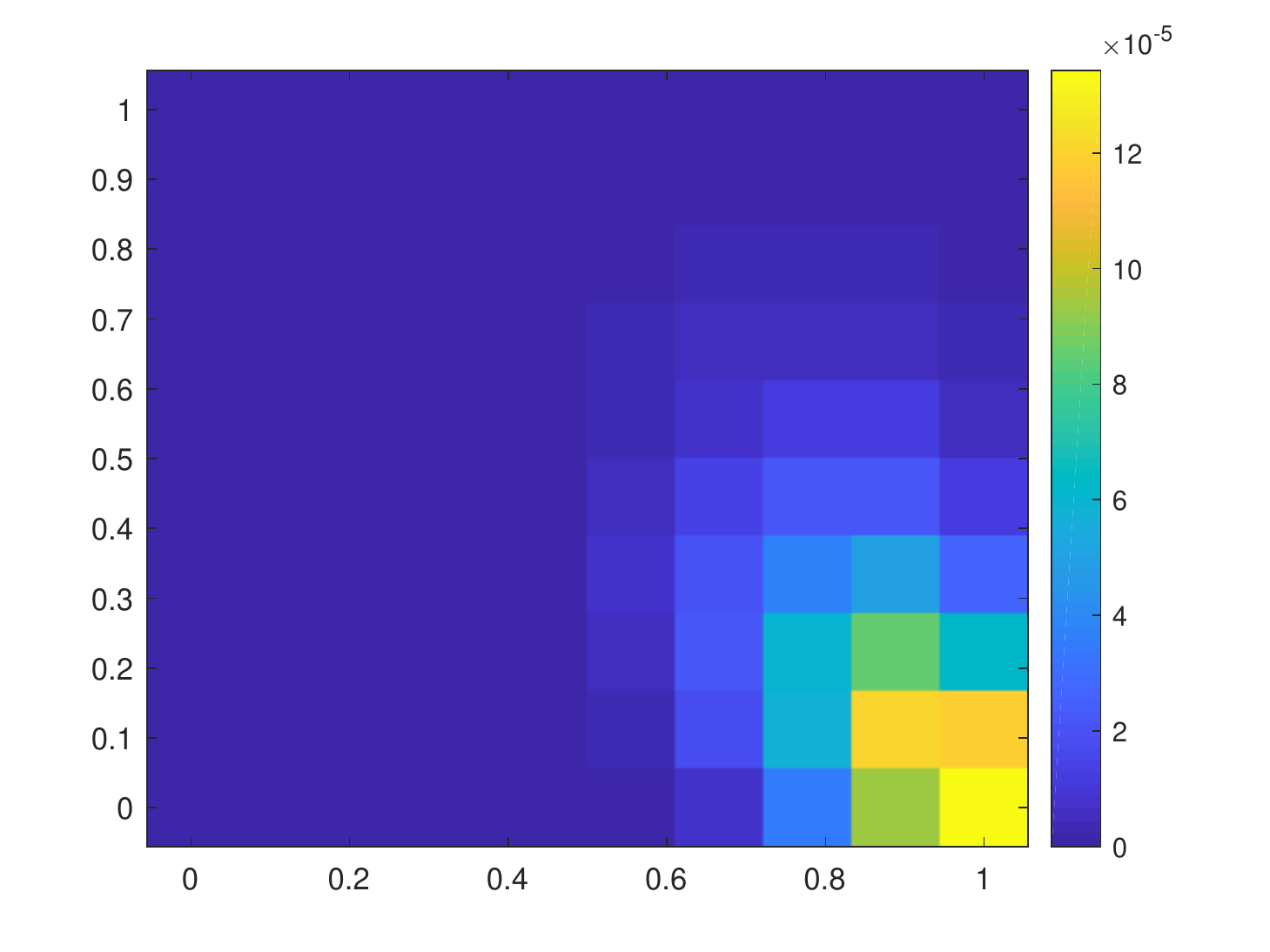}
\includegraphics[scale=0.3]{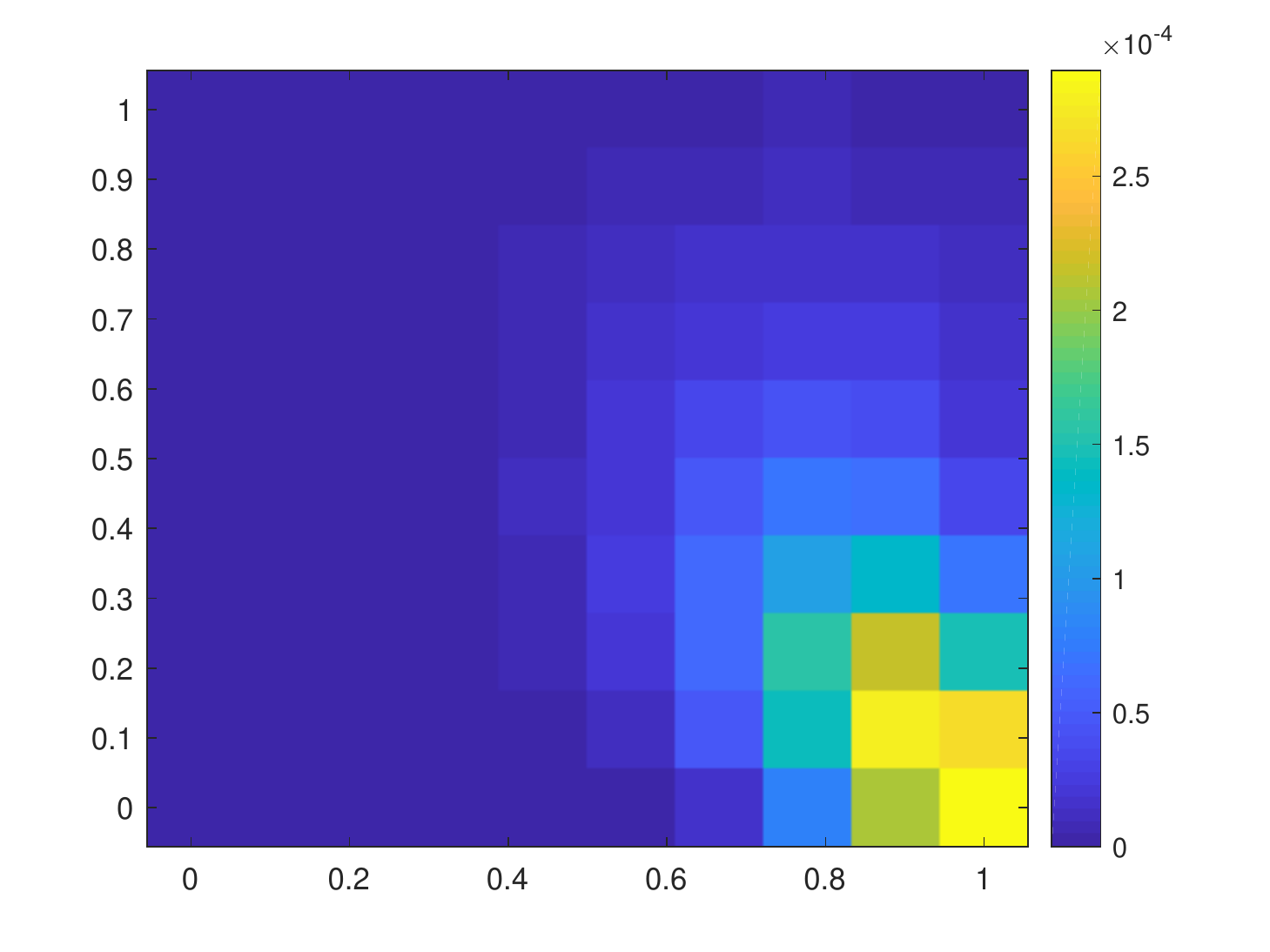}

\includegraphics[scale=0.3]{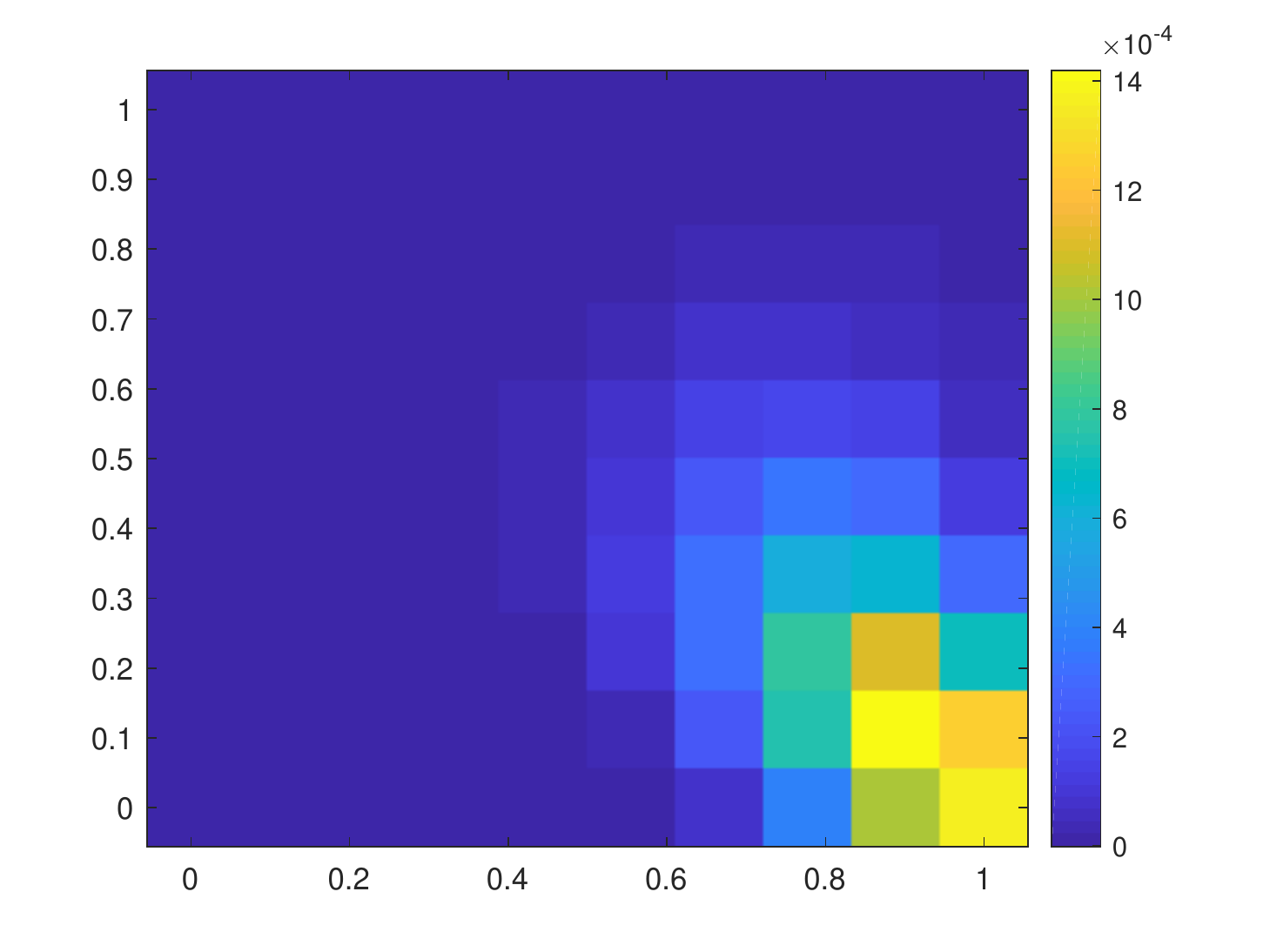} \includegraphics[scale=0.3]{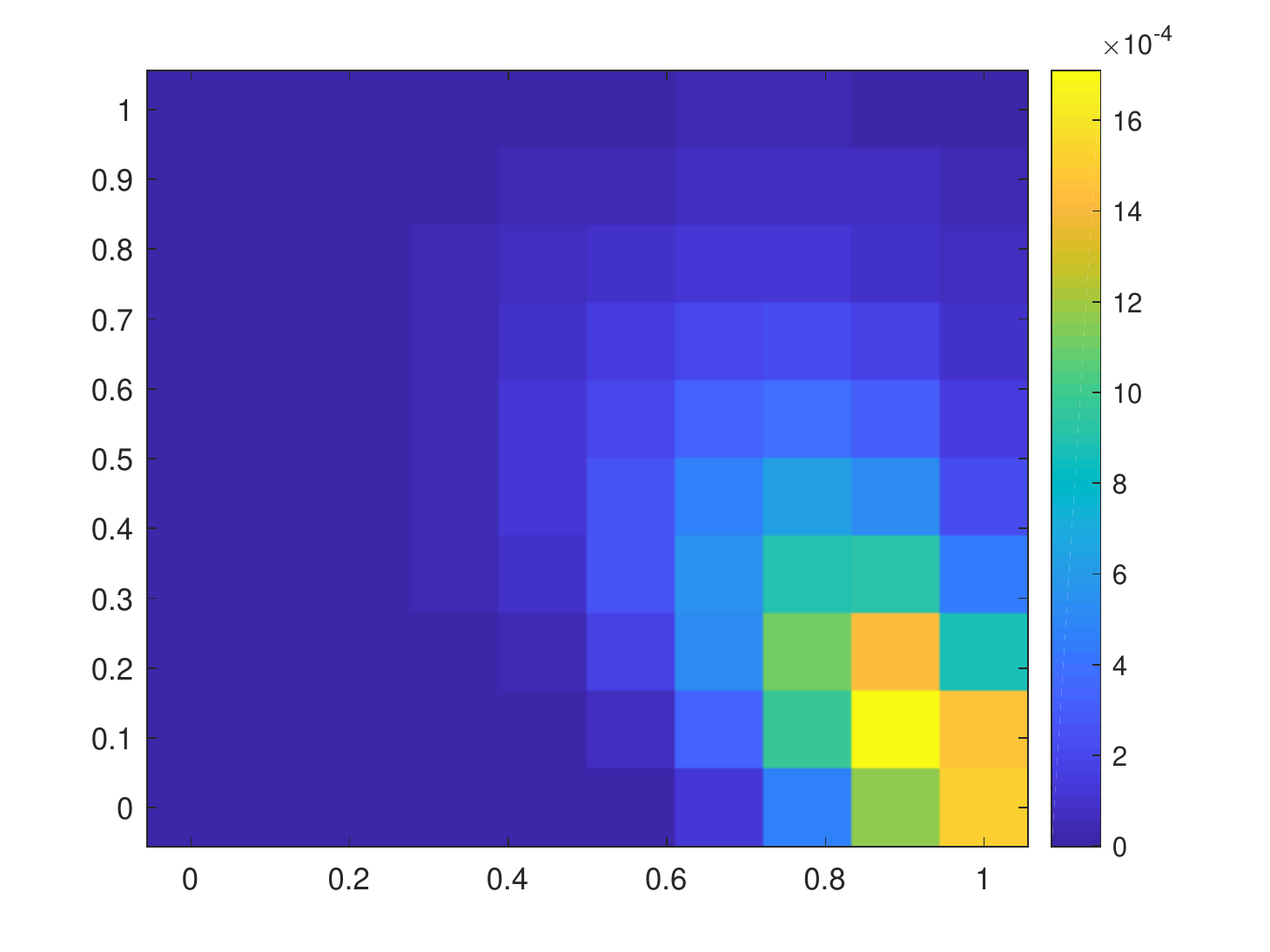}
\includegraphics[scale=0.3]{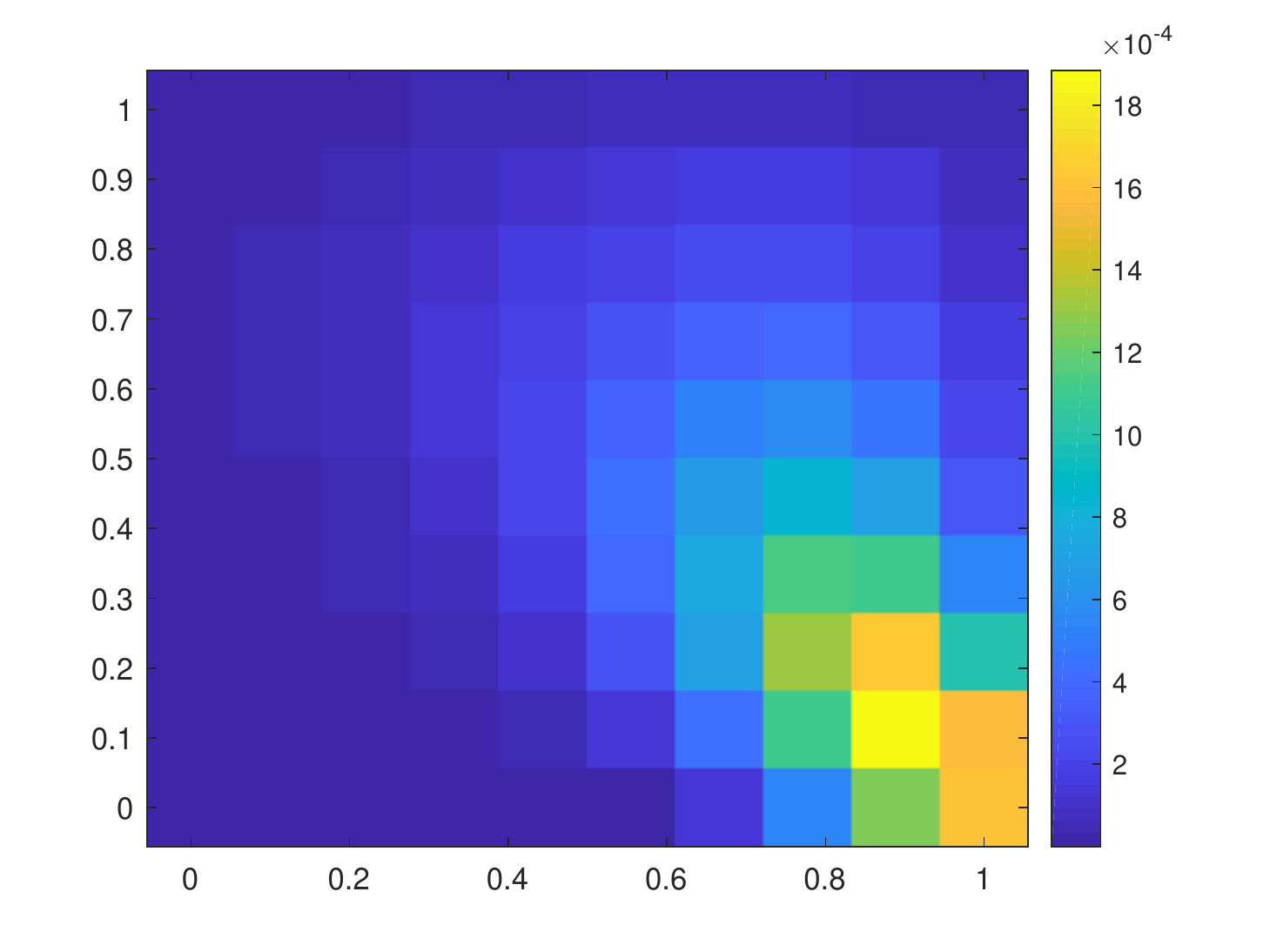}
\caption{Computational results for the second example. Top Left: upscaled solution at $T=0.005$ (matrix), Top Middle:  upscaled solution at $T=0.01$ (matrix),
Top Right:  upscaled solution at $T=0.02$ (matrix),
Bottom Left: upscaled solution at $T=0.005$ (channel), Bottom Middle:  upscaled solution at $T=0.01$ (channel),
Bottom Right:  upscaled solution at $T=0.02$ (channel).}
\label{fig:results}
\end{figure}

\subsection{Nonlinear example with machine learning}

In this section, we present numerical results for the proposed method. 
Our examples use some of the tools developed in
\cite{leung2019space,vasilyeva2019learning}. The goal of this example
is to use machine learning to compute macroscale parameters based on RVE
simulations. The method is similar to our previous approach 
\cite{leung2019space,vasilyeva2019learning}; however, the calculations
are performed in RVEs.

We consider nonlinear flow problem (unsaturated flow problem)  in fractured media $\Omega = [0, 1] \times  [0, 1]$ with no flux boundary conditions. 
We set source terms $q_f = 10^5$  in the fracture continuum in  $[0.95, 1.0] \times [0.95, 1.0]$ . 
We use $10 \times 10$ coarse grid. %and $160 \times 160$ fine grid. 
For the nonlinear coefficient, we use $k^{\alpha}(x, u) =  k_r(u) k^{\alpha}_s(x)$ with $k_r(u) = \exp(-a |u|)$, $a = 0.1$ ($\alpha = m,f$).
We set $c^m = 1$, $c^f = 1$, $k_s^f = 10^3$, $k_m = 1$ and $T_{max} = 0.025$ with 50 time steps.
The numerical calculations of the effective properties has been implemented with the open-source finite element software PETSc and FEniCS \cite{fenics, logg2012automated, balay2004petsc}. Machine Learning algorithm is implemented using Keras library.

\begin{figure}[h!]
\begin{center}
\includegraphics[width=0.9\linewidth]{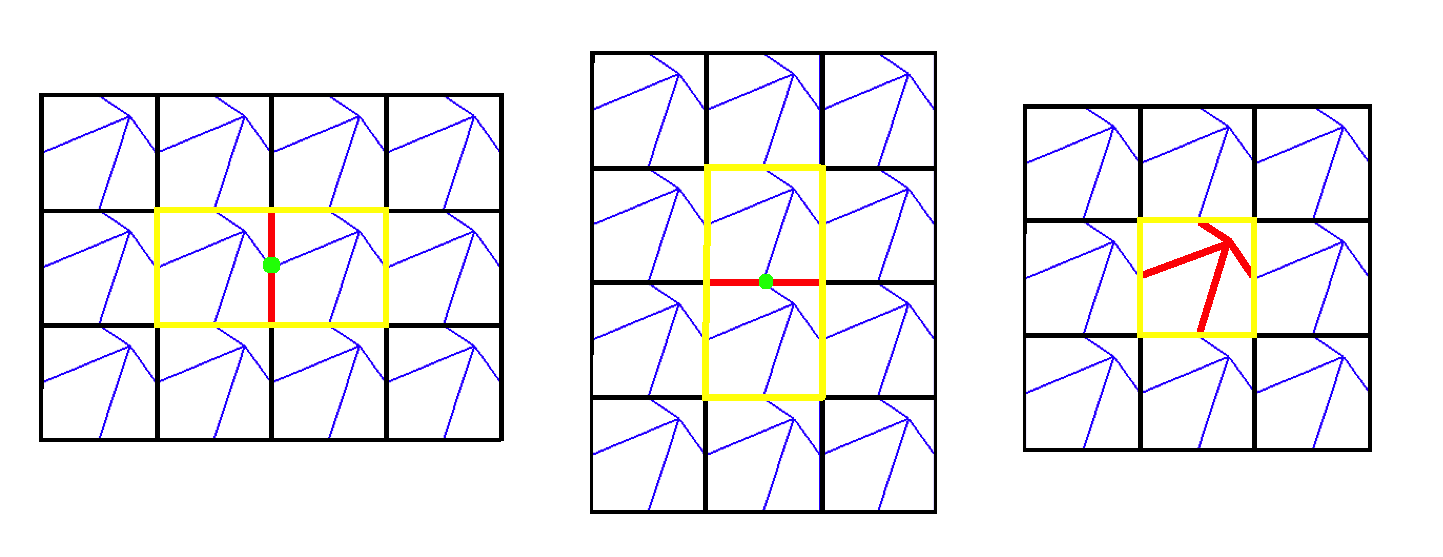}
\end{center}
\caption{Local domains to generate datasets. 
Left: $NN_1$ ($T^{mm}$ on edge (red color)) and $NN_4$ ($T^{ff}$ on point (green color)). 
Left: $NN_2$ ($T^{mm}$ on edge (red color)) and $NN_5$ ($T^{ff}$ on point (green color)).
Left: $NN_3$ ($T^{mf}$ on fracture interface (red color)).}
\label{fig:omega}
\end{figure}

\begin{table}[h!]
\begin{center}
\begin{tabular}{ | c | c c c | }
\hline
& MSE & RMSE  (\%)  & MAE (\%) \\
\hline
$NN_1$ 	& 0.0358 & 1.8929 & 1.5824 \\
$NN_2$ 	& 0.0895 & 2.9930 & 2.4455 \\
$NN_3$ 	& 0.0006 & 0.2544 & 0.2309 \\
$NN_4$ 	& 0.0163 & 1.2801 & 1.2313 \\
$NN_5$ 	& 0.0114 & 1.0686 & 0.9461 \\
\hline
\end{tabular}
\end{center}
\caption{Learning performance of machine learning algorithm }
\label{tab:ml}
\end{table}

\begin{figure}[h!]
\begin{center}
\includegraphics[width=0.32\linewidth]{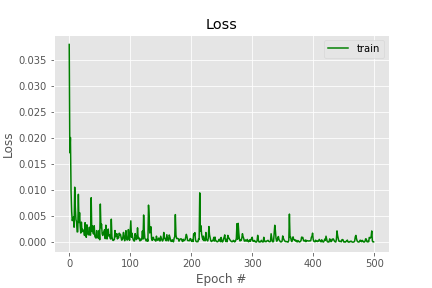}
\includegraphics[width=0.32\linewidth]{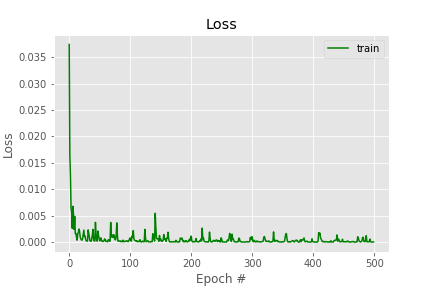}
\includegraphics[width=0.32\linewidth]{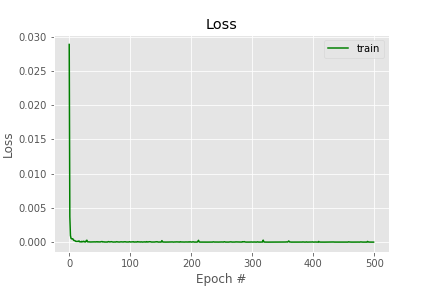}
\includegraphics[width=0.32\linewidth]{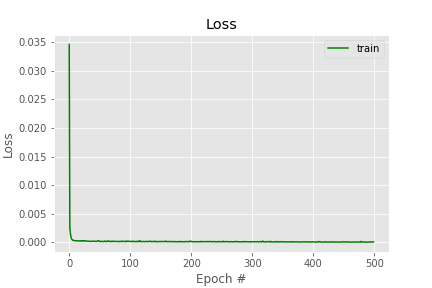}
\includegraphics[width=0.32\linewidth]{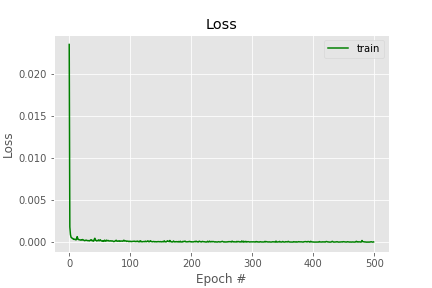}
\end{center}
\caption{Learning performance of machine learning algorithm. Loss functions. 
Left: $NN_1$ ($T^{mm}$) and $NN_4$ ($T^{ff}$. 
Left: $NN_2$ ($T^{mm}$) and $NN_5$ ($T^{ff}$).
Left: $NN_3$ ($T^{mf}$).}
\label{fig:ml}
\end{figure}

\begin{figure}[h!]
\centering
\includegraphics[width=0.49\linewidth]{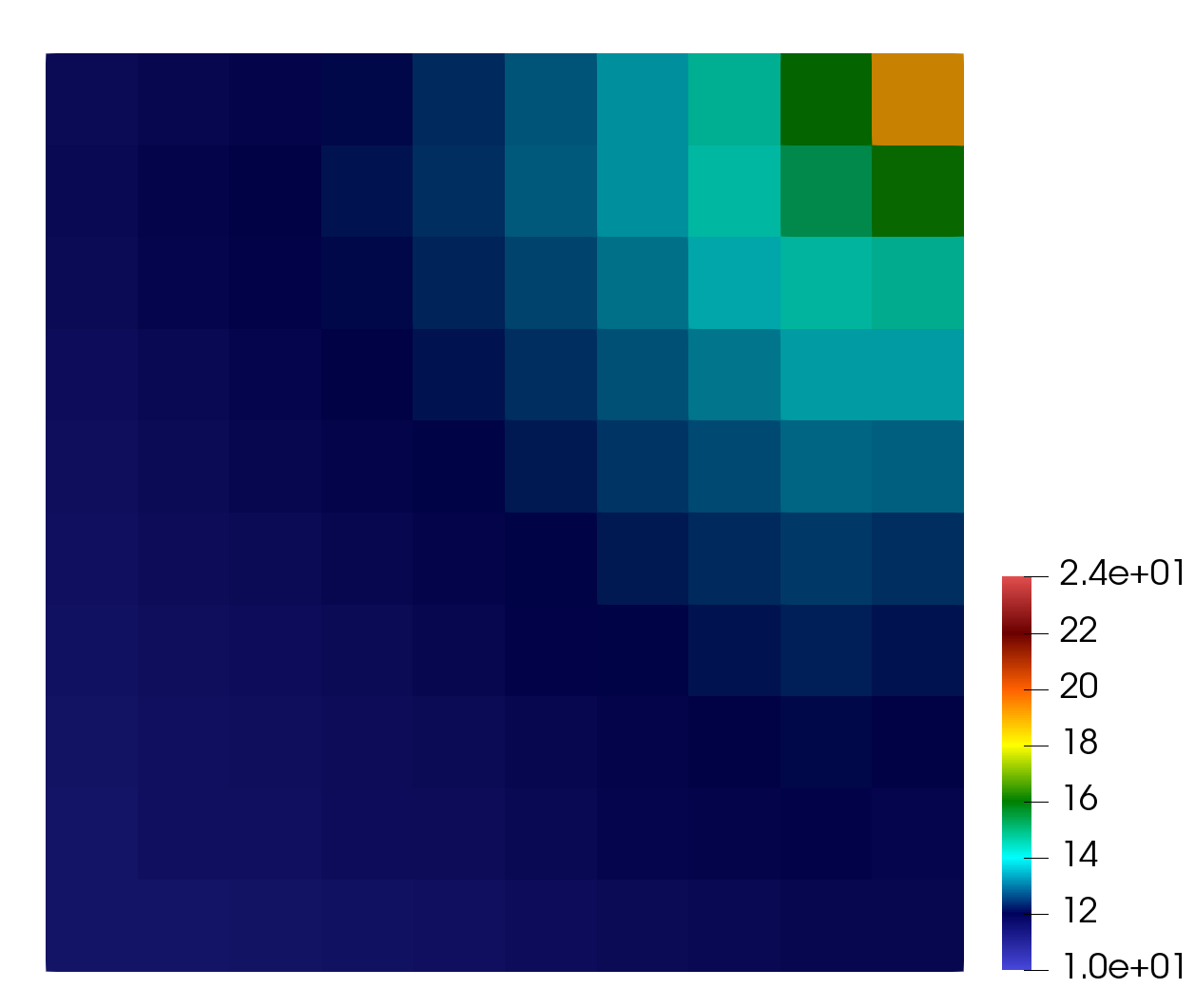}
\includegraphics[width=0.49\linewidth]{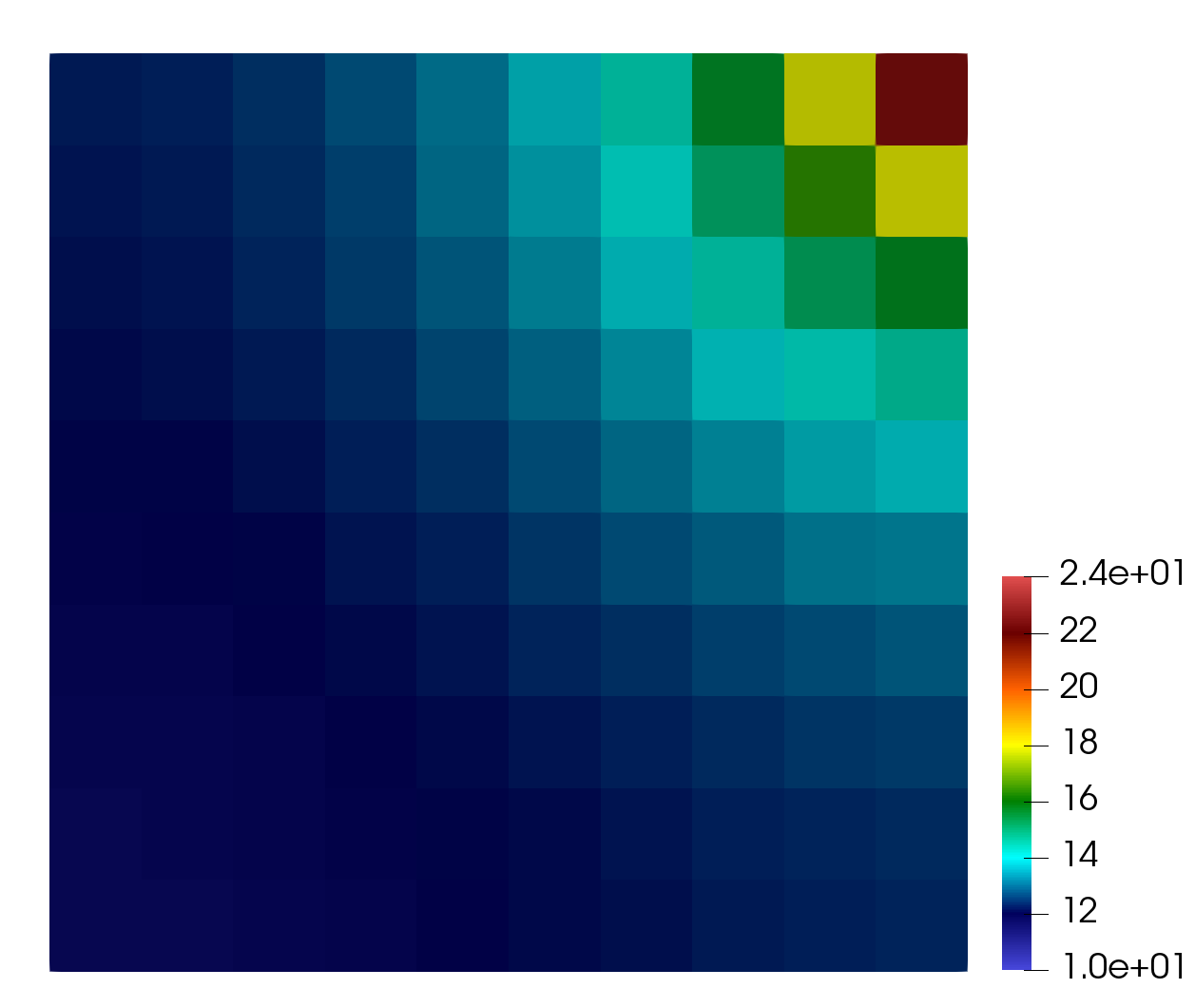}
\caption{Solution on final time $t_m$, $m = 50$. Left: $u_m$. Right: $u_f$}
\label{fig:uu}
\end{figure}

Each sample $X_l$ contains information about coarse grid solution in oversampled local domain
\[
X_l = (X_{l+}^{\overline{u}^{m}}, X_{l+}^{\overline{u}^{f}}),
\]
and output
\[
\text{\textit{Test 1}:} \, 
Y_l = (T_l^{\alpha \beta, NL} ), \quad \alpha,\beta = m, f.
\]

For the training of the neural networks, we use a dataset generated via solution of the local problems in oversamples local domains (see Figure \ref{fig:omega}).   
We train four neural networks for each type of transmissibility: $NN_1$ for horizontal coarse edges for matrix-matrix flow, $NN_2$ for vertical coarse edges s for matrix-matrix flow, $NN_3$ for matrix - fracture flow and $NN_4$ and $NN_5$ for fracture - fracture flow (Figure \ref{fig:omega}).
For calculations, we use 500 epochs with a batch size $N_b = 100$ and Adam optimizer with learning rate $\epsilon = 0.001$.
For accelerating of the training process of the multi-input CNN, we use GPU (GeForce GTX 1060).
We use  $3 \times 3$ convolutions with RELU activation. For each input data, we have 3 layers of CNN with two final fully connected layer. Convolution layer contains  4, 8 and 16 feature maps . We use dropout with rate 10 \% in each layer in order to prevent over-fitting. 
Finally, we combine CNN output and perform three additional fully connected layers with size 200, 50 and 1(one final output).
Presented algorithm is used to learn dependence between multi-input data and upscaled nonlinear transmissibilities.

For error calculation on the dataset, we used  mean square errors, relative mean absolute and relative root mean square errors
\[
MSE = \sum_i |Y_i - \tilde{Y}_i|^2,
\quad
RMSE = \sqrt{ \frac{\sum_i |Y_i - \tilde{Y}_i|^2 }{\sum_i |Y_i|^2 } },
\quad
MAE = \frac{\sum_i |Y_i - \tilde{Y}_i|  }{\sum_i |Y_i|},
\]
where $Y_i$  and $\tilde{Y}_i$ denotes reference and predicted values for sample $X_i$ 
Learning performance for neural networks are presented in Table \ref{tab:ml}. Loss function (MSE) is presented in Figure \ref{fig:ml}. 
We observe a good convergence with small error for each neural network. 
In Figure \ref{fig:uu}, we depict solution of the problem, $u_m$ and $u_f$.

\section*{Acknowledgements}

The research of Eric Chung is partially supported by the Hong Kong RGC General Research Fund (Project numbers 14304217 and 14302018)
and CUHK Faculty of Science Direct Grant 2018-19.
YE would like to thank the partial support from NSF 1620318. YE would also like to acknowledge the support of Mega-grant of the Russian Federation Government (N 14.Y26.31.0013)

\bibliographystyle{plain}
\bibliography{references,references1,references2}

\begin{thebibliography}{10}

\bibitem{abe07}
A.~Abdulle and B.~Engquist.
\newblock Finite element heterogeneous multiscale methods with near optimal
  computational complexity.
\newblock {\em SIAM J. Multiscale Modeling and Simulation}, 6(4):1059--1084,
  2007.

\bibitem{balay2004petsc}
Satish Balay, Kris Buschelman, Victor Eijkhout, William~D Gropp, Dinesh
  Kaushik, Matthew~G Knepley, Lois~Curfman McInnes, Barry~F Smith, and Hong
  Zhang.
\newblock Petsc users manual.
\newblock Technical report, Technical Report ANL-95/11-Revision 2.1. 5, Argonne
  National Laboratory, 2004.

\bibitem{brown2013efficient}
Donald~L Brown, Yalchin Efendiev, and Viet~Ha Hoang.
\newblock An efficient hierarchical multiscale finite element method for stokes
  equations in slowly varying media.
\newblock {\em Multiscale Modeling \& Simulation}, 11(1):30--58, 2013.

\bibitem{cances2015embedded}
Eric Cances, Virginie Ehrlacher, Fr{\'e}d{\'e}ric Legoll, and Benjamin Stamm.
\newblock An embedded corrector problem to approximate the homogenized
  coefficients of an elliptic equation.
\newblock {\em Comptes Rendus Mathematique}, 353(9):801--806, 2015.

\bibitem{chen2019homogenize}
Jie Chen, Shuyu Sun, and Zhengkang He.
\newblock Homogenize coupled stokes--cahn--hilliard system to darcy's law for
  two-phase fluid flow in porous medium by volume averaging.
\newblock {\em Journal of Porous Media}, 22(1), 2019.

\bibitem{chen2019homogenization}
Jie Chen, Shuyu Sun, and Xiaoping Wang.
\newblock Homogenization of two-phase fluid flow in porous media via volume
  averaging.
\newblock {\em Journal of Computational and Applied Mathematics}, 353:265--282,
  2019.

\bibitem{MixedGMsFEM}
E.~Chung, Y.~Efendiev, and C.~Lee.
\newblock Mixed generalized multiscale finite element methods and applications.
\newblock {\em SIAM Multicale Model. Simul.}, 13:338--366, 2014.

\bibitem{WaveGMsFEM}
E.~Chung, Y.~Efendiev, and W.~T. Leung.
\newblock Generalized multiscale finite element method for wave propagation in
  heterogeneous media.
\newblock {\em SIAM Multicale Model. Simul.}, 12:1691--1721, 2014.

\bibitem{chung2018constraintmixed}
Eric Chung, Yalchin Efendiev, and Wing~Tat Leung.
\newblock Constraint energy minimizing generalized multiscale finite element
  method in the mixed formulation.
\newblock {\em Computational Geosciences}, 22(3):677--693, 2018.

\bibitem{NLMC}
Eric~T Chung, Efendiev, Wing~Tat Leung, Maria Vasilyeva, and Yating Wang.
\newblock Non-local multi-continua upscaling for flows in heterogeneous
  fractured media.
\newblock {\em arXiv preprint arXiv:1708.08379}, 2018.

\bibitem{chung2018nonlinear}
Eric~T Chung, Yalchin Efendiev, Wing~T Leung, and Mary Wheeler.
\newblock Nonlinear nonlocal multicontinua upscaling framework and its
  applications.
\newblock {\em International Journal for Multiscale Computational Engineering},
  16(5), 2018.

\bibitem{chung2018constraintCMAME}
Eric~T Chung, Yalchin Efendiev, and Wing~Tat Leung.
\newblock Constraint energy minimizing generalized multiscale finite element
  method.
\newblock {\em Computer Methods in Applied Mechanics and Engineering},
  339:298--319, 2018.

\bibitem{chung2018fast}
Eric~T Chung, Yalchin Efendiev, and Wing~Tat Leung.
\newblock Fast online generalized multiscale finite element method using
  constraint energy minimization.
\newblock {\em Journal of Computational Physics}, 355:450--463, 2018.

\bibitem{chung2015goal}
Eric~T Chung, Wing~Tat Leung, and Sara Pollock.
\newblock Goal-oriented adaptivity for {GM}s{FEM}.
\newblock {\em Journal of Computational and Applied Mathematics}, pages
  625--637, 2015.

\bibitem{ee03}
W.~E and B.~Engquist.
\newblock Heterogeneous multiscale methods.
\newblock {\em Comm. Math. Sci.}, 1(1):87--132, 2003.

\bibitem{GMsFEM13}
Y.~Efendiev, J.~Galvis, and {T. Y.} Hou.
\newblock Generalized multiscale finite element methods (gmsfem).
\newblock {\em Journal of Computational Physics}, 251:116--135, 2013.

\bibitem{fafalis2018computational}
Dimitrios Fafalis and Jacob Fish.
\newblock Computational continua for linear elastic heterogeneous solids on
  unstructured finite element meshes.
\newblock {\em International Journal for Numerical Methods in Engineering},
  115(4):501--530, 2018.

\bibitem{fish2010computational}
Jacob Fish and Sergey Kuznetsov.
\newblock Computational continua.
\newblock {\em International Journal for Numerical Methods in Engineering},
  84(7):774--802, 2010.

\bibitem{fish2005multiscale}
Jacob Fish and Zheng Yuan.
\newblock Multiscale enrichment based on partition of unity.
\newblock {\em International Journal for Numerical Methods in Engineering},
  62(10):1341--1359, 2005.

\bibitem{fu2019edge}
Shubin Fu, Eric Chung, and Guanglian Li.
\newblock Edge multiscale methods for elliptic problems with heterogeneous
  coefficients.
\newblock {\em Journal of Computational Physics}, 2019.

\bibitem{gao2015generalized}
Kai Gao, Shubin Fu, Richard~L Gibson~Jr, Eric~T Chung, and Yalchin Efendiev.
\newblock Generalized multiscale finite-element method (gmsfem) for elastic
  wave propagation in heterogeneous, anisotropic media.
\newblock {\em Journal of Computational Physics}, 295:161--188, 2015.

\bibitem{hkj12}
H.~Hajibeygi, D.~Kavounis, and P.~Jenny.
\newblock A hierarchical fracture model for the iterative multiscale finite
  volume method.
\newblock {\em Journal of Computational Physics}, 230(4):8729--8743, 2011.

\bibitem{henning2012localized}
Patrick Henning, Axel Malqvist, and Daniel Peterseim.
\newblock A localized orthogonal decomposition method for semi-linear elliptic
  problems.
\newblock {\em arXiv preprint arXiv:1211.3551}, 2012.

\bibitem{hs05}
V.H. Hoang and C.~Schwab.
\newblock High dimensional finite elements for elliptic problems with multiple
  scales.
\newblock {\em SIAM Multiscale Modeling and Simulation}, 3:168--194, 2004.

\bibitem{hw97}
T.~Hou and X.H. Wu.
\newblock A multiscale finite element method for elliptic problems in composite
  materials and porous media.
\newblock {\em J. Comput. Phys.}, 134:169--189, 1997.

\bibitem{jennylt03}
P.~Jenny, S.H. Lee, and H.~Tchelepi.
\newblock Multi-scale finite volume method for elliptic problems in subsurface
  flow simulation.
\newblock {\em J. Comput. Phys.}, 187:47--67, 2003.

\bibitem{jennylt05}
P.~Jenny, S.H. Lee, and H.~Tchelepi.
\newblock Adaptive multi-scale finite volume method for multi-phase flow and
  transport in porous media.
\newblock {\em SIAM J. Multiscale Modeling and Simulation}, 3:30--64, 2004.

\bibitem{Jikov91}
V.~V. Jikov, S.~M. Kozlov, and O.~A. Oleinik.
\newblock {\em Homogenization of Differential Operators and Integral
  Functionals}.
\newblock Springer-Verlag, 1991.

\bibitem{le2014msfem}
Claude Le~Bris, Fr{\'e}d{\'e}ric Legoll, and Alexei Lozinski.
\newblock An msfem type approach for perforated domains.
\newblock {\em Multiscale Modeling \& Simulation}, 12(3):1046--1077, 2014.

\bibitem{le2014multiscale}
Claude Le~Bris, Fr{\'e}d{\'e}ric Legoll, and Florian Thomines.
\newblock Multiscale finite element approach for weakly random problems and
  related issues.
\newblock {\em ESAIM: Mathematical Modelling and Numerical Analysis},
  48(3):815--858, 2014.

\bibitem{leung2019space}
Wing~T Leung, Eric~T Chung, Yalchin Efendiev, Maria Vasilyeva, and Mary
  Wheeler.
\newblock Space-time nonlinear upscaling framework using non-local
  multi-continuum approach.
\newblock {\em arXiv preprint arXiv:1908.05582, to appear in International
  Journal of Multiscale Engineering}, 2019.

\bibitem{fenics}
Anders Logg, Kent-Andre Mardal, and Garth Wells.
\newblock {\em Automated solution of differential equations by the finite
  element method: The FEniCS book}, volume~84.
\newblock Springer Science \& Business Media, 2012.

\bibitem{logg2012automated}
Anders Logg, Kent-Andre Mardal, and Garth Wells.
\newblock {\em Automated solution of differential equations by the finite
  element method: The FEniCS book}, volume~84.
\newblock Springer Science \& Business Media, 2012.

\bibitem{oz06_1}
H.~Owhadi and L.~Zhang.
\newblock Metric-based upscaling.
\newblock {\em Comm. Pure. Appl. Math.}, 60:675--723, 2007.

\bibitem{rk07}
A.J. Roberts and I.~Kevrekidis.
\newblock General tooth boundary conditions for equation free modeling.
\newblock {\em SIAM J. Sci. Comput.}, 29(4):1495--1510, 2007.

\bibitem{salama2017flow}
Amgad Salama, Shuyu Sun, Mohamed~F El~Amin, Yi~Wang, and Kundan Kumar.
\newblock Flow and transport in porous media: A multiscale focus.
\newblock {\em Geofluids}, 2017, 2017.

\bibitem{skr06}
G.~Samaey, I.G. Kevrekidis, and D.~Roose.
\newblock Patch dynamics with buffers for homogenization problems.
\newblock {\em J. Comput. Phys.}, 213(1):264--287, 2006.

\bibitem{srk05}
G.~Samaey, D.~Roose, and I.G. Kevrekidis.
\newblock The gap-tooth scheme for homogenization problems.
\newblock {\em SIAM J. Multiscale Modeling and Simulation}, 4(1):278--306,
  2005.

\bibitem{tan2019high}
Wee~Chin Tan and Viet~Ha Hoang.
\newblock High dimensional finite element method for multiscale nonlinear
  monotone parabolic equations.
\newblock {\em Journal of Computational and Applied Mathematics}, 345:471--500,
  2019.

\bibitem{vasilyeva2019learning}
Maria Vasilyeva, Wing~T Leung, Eric~T Chung, Yalchin Efendiev, and Mary
  Wheeler.
\newblock Learning macroscopic parameters in nonlinear multiscale simulations
  using nonlocal multicontinua upscaling techniques.
\newblock {\em arXiv preprint arXiv:1907.02921}, 2019.

\end{thebibliography}

\end{document}